\documentclass[11pt]{article}

\usepackage{amsmath,amssymb,amsthm}
\usepackage[utf8]{inputenc}
\usepackage{centernot}
\usepackage{verbatim}
\usepackage{dsfont}
\usepackage{tikz}
\usepackage{color}

\theoremstyle{plain}
\newtheorem{theorem}{Theorem}[section]

\newtheorem{lemma}[theorem]{Lemma}
\newtheorem{corollary}[theorem]{Corollary}

\theoremstyle{definition}
\newtheorem{definition}[theorem]{Definition}
\newtheorem{remark}[theorem]{Remark}
\newtheorem{example}[theorem]{Example}
\newtheorem{question}[theorem]{Question}

\newcommand{\N}{\mathbb{N}} 
\newcommand{\Z}{\mathbb{Z}} 
\newcommand{\R}{\mathbb{R}} 
\newcommand{\C}{\mathbb{C}} 
\newcommand{\D}{\mathbb{D}} 
\newcommand{\T}{\mathbb{T}} 

\begin{document}

\title{Frequently recurrent operators}
\author{Antonio Bonilla, Karl-G. Grosse-Erdmann, Antoni L\'opez-Mart\'{\i}nez, Alfred Peris\thanks{The work was partially supported by MCIN/AEI/10.13039/501100011033, Projects PID2019-105011GB-I00  and MTM2016-75963-P; the second author was also supported by FNRS Grant PDR T.0164.16; the third author was supported the Ministry of Education of Spain (grant FPU2019/04094); the fourth author was also supported by Generalitat Valenciana, Projects PROMETEO/2017/102 and  PROMETEU/2021/070.}}
\date{}



\maketitle

\begin{abstract}
Motivated by a recent investigation of Costakis et al. on the notion of recurrence in linear dynamics, we study various stronger forms of recurrence for linear operators, in particular that of frequent recurrence. We study, among other things, the relationship between a type of recurrence and the corresponding notion of hypercyclicity, the influence of power boundedness, and the interplay between recurrence and spectral properties. We obtain, in particular, Ansari- and Léon-M\"uller-type theorems for $\mathcal{F}$-recurrence under very weak assumptions on the Furstenberg family $\mathcal{F}$. This allows us, as a by-product, to deduce Ansari- and Léon-M\"uller-type theorems for $\mathcal{F}$-hypercyclicity.
\end{abstract}

\section{Introduction}

The notion of recurrence for a dynamical system has a very long history, whose systematic study goes back to the work of Gottschalk and Hedlund \cite{GoHe} and Furstenberg \cite{Fur3} (see also \cite{Glas} and \cite{OZ} for recent advances). In linear dynamics, however, recurrent operators have only recently been studied systematically in a fundamental paper by Costakis, Manoussos and Parissis \cite{CMP}; see also \cite{CP}.

The literature on (non-linear) dynamical systems abounds with notions that are similar to recurrence. Of course, periodicity is a very strong form of recurrence, and it is fundamental in any dynamical theory. But some other forms of recurrence have also recently been looked at in linear dynamics, see \cite{GMOP}, \cite{YW}, \cite{HHY18}, \cite{GriMaMe}, \cite{CaMu}.

The aim of this paper is to study various notions of recurrence in the context of linear dynamics. The appropriate framework is that of $\mathcal{F}$-recurrence for arbitrary Furstenberg families $\mathcal{F}$. However, for better readability we will mainly concentrate on those types of recurrence that deserve the greatest interest from the point of view of linear dynamics. We will discuss the general notion of $\mathcal{F}$-recurrence in Section \ref{s-frec}.

Throughout Sections 1 to 7, $X$ will denote a Fr\'echet space and $T:X\rightarrow X$ a (continuous, linear) operator, briefly $T\in L(X)$. A vector $x\in X$ is called recurrent for $T$ if there exists a strictly increasing sequence $(n_k)_{k\in \mathbb{N}}$ of positive integers such that
\[
T^{n_k}x\rightarrow x\quad \text{as $k\rightarrow\infty$.}
\]
We will denote by $\text{Rec}(T)$ the set of recurrent vectors for $T$, and $T$ is called recurrent if $\text{Rec}(T)$ is dense in $X$.
The latter differs from, but is equivalent to the definition of recurrence given by Costakis et al., see \cite[Proposition 2.1 with Remark 2.2]{CMP} and Remark \ref{r-topfrec} below.

A vector $x$ is called periodic for $T$ if there is some $n\geq 1$ such that $T^nx=x$. The set of periodic points of $T$ will be denoted by $\text{Per}(T)$. The vector $x$ is called uniformly recurrent for $T$ if, for any neighbourhood $U$ of $x$, the return set
\[
N(x,U)=\{n\geq 0 : T^nx\in U\}
\]
is syndetic, that is, has bounded gaps. The set of uniformly recurrent vectors will be denoted by $\text{URec}(T)$. Uniformly recurrent vectors are often called almost periodic in the literature, see \cite{GoHe}, but also syndetically recurrent or strongly recurrent, see \cite{BK04}, \cite{KoSn09}.

In addition, we fix the following terminology as suggested by recent work in linear dynamics.

\begin{definition}
Let $T\in L(X)$. A vector $x\in X$ is called \textit{frequently recurrent} (\textit{upper frequently recurrent}, \textit{reiteratively recurrent}) for $T$ if, for any neighbourhood $U$ of $x$, the return set
\[
N(x,U)=\{n\geq 0 : T^nx\in U\}
\]
has positive lower density (positive upper density, positive upper Banach density, respectively). The corresponding set of vectors is denoted by $\text{FRec}(T)$ ($\text{UFRec}(T)$, $\text{RRec}(T)$, respectively). If this set is dense in $X$ then the operator is called \textit{frequently recurrent} (\textit{upper frequently recurrent}, \textit{reiteratively recurrent}, respectively).
\end{definition}

We recall that, for a subset $A$ of $\mathbb{N}_0$, its lower density is defined as
\[
\underline{\mbox{dens}}\,(A) = \liminf_{N\to\infty}\frac{\text{card} \{n\in A : n\leq N\}}{N+1},
\]
its upper density as
\[
\overline{\mbox{dens}}\,(A) = \limsup_{N\to\infty}\frac{\text{card} \{n\in A : n\leq N\}}{N+1},
\]
and its upper Banach density as
\[
\overline{\mbox{Bd}}\,(A) = \lim_{N\to\infty}\sup_{m\geq 0}\frac{\text{card} \{n\in
A : m\le n\le m+N\}}{N+1};
\]
see \cite{GTT10} for other, equivalent definitions of the upper Banach density; see also \cite{BG}.

The notion of upper frequent recurrence was first introduced by Costakis and Parissis \cite{CP}, while Grivaux and Matheron \cite{GrMa14} have introduced a concept of frequent recurrence that is (at least formally) weaker than ours.

In non-linear dynamics, frequently recurrent points have been called weakly almost periodic, upper frequently recurrent points have been called quasi-weakly almost periodic or ergodic, and reiteratively recurrent points have been called positive Banach upper density points, Banach recurrent points, or essentially recurrent points, see \cite{HeYiZh13}, \cite{Li12}, \cite{YZ12}, \cite{WaHu18}, \cite{BeDo08}.

As pointed out by the referee, there is an important notion of (topological) multiple recurrence studied for linear operators in \cite{CP}. It was observed in the recent article \cite{CarMur} that it is equivalent to $\mathcal{AP}$-recurrence, where $\mathcal{AP}$ is the Furstenberg family consisting of those subsets of $\mathbb{N}_0$ containing arbitrarily long arithmetic progressions. Let $\text{APRec}(T)$ denote the set of $\mathcal{AP}$-recurrent vectors for an operator $T$.

We have the following inclusions, which are obvious, except $\text{RRec}(T) \subset \text{APRec}(T)$, which was observed in \cite{CarMur}.
\begin{equation}\label{eq-incl}
\text{Per}(T)\subset \text{URec}(T)\subset \text{FRec}(T)\subset \text{UFRec}(T)\subset \text{RRec}(T) \subset \text{APRec}(T)\subset \text{Rec}(T).
\end{equation}

The paper is organized as follows. In Section \ref{s-recsize} we compare recurrence properties with their corresponding notions of hypercyclicity. In Section \ref{s-recpower} we study the influence of power boundedness on recurrence. Section \ref{s-recspec} is devoted to some structural properties of recurrence; in particular, we solve a problem of Grivaux et al.~\cite{GriMaMe}. Weighted backward shifts are studied in Section \ref{s-recshift}, where we also show that all the inclusions in \eqref{eq-incl} are strict (in a rather strong sense). Further operators are considered in Section \ref{s-receigen}; a common feature of many of these operators is a large supply of unimodular eigenvectors, which implies $\text{IP}^*$-recurrence, an interesting strengthening of uniform recurrence. Thus, as a preparation, we briefly discuss  $\text{IP}^*$-recurrence in Section~\ref{s-iprec}. In the final Section \ref{s-frec} we introduce and discuss the general notion of $\mathcal{F}$-recurrence for operators on general topological vector spaces. As a by-product of our work we obtain Ansari- and Le\'on-M\"uller-type results for $\mathcal{F}$-hypercyclicity, see Theorem~\ref{t-AnsLeMuHyp}.

Our investigations have led to several open problems, see Questions \ref{q2}, \ref{q3},  \ref{q4}, \ref{q5}, \ref{q8}\footnote{Question \ref{q8} has recently been solved in the negative; see \cite{CarMur2}.}, \ref{q9}, and \ref{q10}.

For any unexplained, but standard notions from linear dynamics we refer to the textbooks \cite{BaMa09} and \cite{GrPe11}.

\section{Recurrence, hypercyclicity, and the size of the set of recurrent vectors}\label{s-recsize}

The central notion in linear dynamics is that of a hypercyclic vector: it is characterized by having a dense orbit. In a similar vein,  a vector $x\in X$ is called frequently hypercyclic (upper frequently hypercyclic, reiteratively hypercyclic) for $T$ if, for every non-empty open subset $U$ of $X$, the return set $N(x,U)$ has positive lower density (positive upper density, positive upper Banach density, respectively). An operator that possesses such a vector is called frequently hypercyclic (upper frequently hypercyclic, reiteratively hypercyclic, respectively), see \cite{BaGr06}, \cite{Shk09}, \cite{BMPP16}, \cite{BG}, and the textbooks \cite{BaMa09}, \cite{GrPe11}. Note that uniform recurrence admits no hypercyclic analogue, see \cite[Proposition 2]{BMPP16}.

Trivially, every notion of hypercyclicity implies the corresponding notion of recurrence. The converse, of course, is not true as seen by the identity operator. In this section we ask under which additional assumptions on the operator the converse does become true.

Our first result elaborates on \cite[Theorem 14]{BMPP16}.


\begin{theorem}\label{reit-hyp}
Let $T\in L(X)$. Then the following assertions are equivalent:
\begin{itemize}	
\item[\rm (a)] $T$ is reiteratively hypercyclic;	
\item[\rm (b)] $T$ is hypercyclic, and $\text{\emph{RRec}}(T)$ is a residual set;
\item[\rm (c)] $T$ is hypercyclic, and $\text{\emph{RRec}}(T)$ is of second category;
\item[\rm (d)] $T$ admits a hypercyclic reiteratively recurrent vector;	
\item[\rm (e)] $T$ is hypercyclic and reiteratively recurrent;	
\item[\rm (f)] $T$ is hypercyclic, and every hypercyclic vector is reiteratively hypercyclic.
\end{itemize}	
In that case the set of hypercyclic reiteratively recurrent vectors is residual.	
\end{theorem}
	
\begin{proof} (a) $\Rightarrow$ (b). By \cite[Theorem 14]{BMPP16}, if $T$ is reiteratively hypercyclic then every hypercyclic vector is reiteratively hypercyclic; and the set of hypercyclic vectors is residual.

(b) $\Rightarrow$ (c). This is trivial.
	
(c) $\Rightarrow$ (d). This follows from the fact that the set of hypercyclic vectors of a hypercyclic operator is residual.	
	
(d) $\Rightarrow$ (e). Since $T$ admits a hypercyclic reiteratively recurrent vector $x$, then each element of the orbit of $x$ is also a hypercyclic reiteratively recurrent vector. Thus $T$ is hypercyclic and reiteratively recurrent.
	
(e) $\Rightarrow$ (f). This was essentially shown in the proof of \cite[Theorem 14]{BMPP16}. We repeat the argument for the sake of completeness.
	
Let $x$ be a hypercyclic vector and $U$ a non-empty open set. By hypothesis there is a reiteratively recurrent vector $y\in U$. Thus, $N(y,U) = \{ n\ge 0: T^ny \in U \}$ has positive upper Banach density.
		
Now let $n\ge 0$. Then $U_n = \bigcap_{j\in N(y,U)\cap[0,n]}T^{-j}(U)$ is an open set, and it is non-empty since it always contains $y$. By hypercyclicity of $x$ there is then some $k_n\geq 0$ such that $T^{k_n}x \in U_n$, thus $T^{k_n +j}x \in U$ for every $j \in N(y,U)\cap[0,n]$.

In other words, for every $n$ sufficiently large, there exists $k_n\ge 0$ such that
$$
N(x,U) \supset k_n +(N(y,U)\cap[0,n]).
$$		
This easily implies that $N(x,U) $ has positive upper Banach density. That is, $x$ is reiteratively hypercyclic.
		
(f) $\Rightarrow$ (a). This is trivial.		
\end{proof}

In particular, since periodic points are reiteratively recurrent, we have the following result of Menet \cite{Men17}.

\begin{corollary}
Every chaotic operator is reiteratively hypercyclic.
\end{corollary}

In view of the theorem one might wonder whether, for a hypercyclic operator, a single non-zero reiteratively recurrent vector suffices to make it reiteratively hypercyclic. The following example shows that this is not the case.

\begin{example}
By \cite[Theorem 13]{BMPP16}, there exists a mixing operator $S$ on $\ell^2(\N)$ that is not reiteratively hypercyclic, and let $T$ be a mixing and chaotic operator on $\ell^2(\N)$, for example twice the backward shift, $2B$ (see \cite[Example 3.2]{GrPe11}). Then the operator $S\times T$ on $\ell^2(\N)\times\ell^2(\N)$ is also mixing; by the standard quasi-conjugacy argument (see \cite[Proposition 1.42]{GrPe11}) $S\times T$ cannot be reiteratively hypercyclic because $S$ is not; and $(0,y)$ is periodic for $S\times T$ if $y$ is periodic for $T$. So we even have a mixing operator with a non-zero periodic point that is not reiteratively hypercyclic.
\end{example}

If $T$ is recurrent, the set of recurrent vectors for $T$ is residual, see \cite{CMP}. Also, if $T$ is a reiteratively hypercyclic operator, then the set of reiteratively hypercyclic vectors is residual, see \cite{BMPP16}. Thus one would expect the same type of result for reiterative recurrence. Surprisingly, as we see next, this is not the case.

\begin{example}
There is a reiteratively recurrent operator for which the set of reiteratively recurrent vectors is of first category. To see this, let $X=\ell^p(\mathbb{N})$, $1\leq p<\infty$, or $c_0(\mathbb{N})$. We consider the operator $T : X \rightarrow X$ that is defined by $Te_1=e_1$ and
\[
Te_k = \begin{cases} 2 e_{k+1} & \text{if } 2^j \leq k < 2^{j+1}-1, \\
\frac{1}{2^{(2^j-1)}}e_{2^j} & \text{if } k=2^{j+1}-1
\end{cases}
\]
for $j\geq 1$, where $e_k=(\delta_{k,n})_{n\geq 1}$ denotes the $k$-th canonical unit sequence. Obviously, every vector $e_k$, $k\geq 1$, is periodic for $T$, so that $T$ admits a dense set of periodic points. In particular, $T$ is reiteratively recurrent.

We now show that $\text{RRec}(T)$ is of first category. For this, it suffices to show that
\[
G=\{x=(x_n)_{n\geq 1} \in X :|x_{2^j}|>\tfrac{1}{j} \text{ for infinitely many $j\geq 1$}\}
\]
is a residual set that does not contain any reiteratively recurrent vector. In fact, it is easy to see that $G$ is dense, and since
\[
G=\bigcap_{J\geq 1} \bigcup_{j\geq J}\{x\in X :|x_{2^j}|>\tfrac{1}{j}\},
\]
it is a dense $G_\delta$-set, hence residual. On the other hand, suppose that $x\in G \cap \text{RRec}(T)$. Let $U = \{y\in X : \|y-x\| < \frac{1}{2}\}$. Then the set $A=\{n\geq 0 : T^nx\in U\}$ has positive upper Banach density. Thus there is some $\delta>0$ and some $N_0\geq 0$ such that, for any $N\geq N_0$,
\begin{equation}\label{eqLP}
\sup_{m\geq 0}\frac{\text{card} \{n\in [m, m+N] : T^nx\in U\}}{N+1}>\delta.
\end{equation}

Now, since $x=(x_n)_n\in X$, there is some $N_1\geq 0$ such that $|x_n|<\frac{1}{2}$ for every $n\geq N_1$. Consequently we have that if $y \in U$ then $|y_n|<1$ for every $n\geq N_1$. Moreover, since $x \in G$, there is some $j\geq 1$ such that $\frac{j}{2^j}<\frac{\delta}{2}$, $2^j > \max(N_0,N_1)$ and $|x_{2^j}|>\frac{1}{j}$. We then have for any $n=\ell 2^j +k$, $\ell\geq 0$, $j\leq k \leq 2^j-1$,
\[
|[T^n x]_{2^j+k}| = 2^k |x_{2^j}| > \frac{2^k}{j} \geq \frac{2^j}{j} > 1,
\]
so that $T^nx\notin U$. This implies that, for any $m\geq 0$, $\text{card}\{n\in [m,m+2^j-1] : T^nx \in U\} \leq j$, hence
\[
\sup_{m\geq 0}\frac{\text{card}\{n\in [m,m+2^j-1] : T^nx \in U\}}{2^j}\leq \frac{j}{2^j}<\frac{\delta}{2};
\]
since $2^j > N_0$, this contradicts \eqref{eqLP}.
\end{example}

We will now see that for upper frequently recurrent operators the situation is a little different from that for reiterative recurrence found in Theorem \ref{reit-hyp}. We start with a partial analogue.

\begin{theorem}\label{ufreq-hyp}
Let $T\in L(X)$. Then the following assertions are equivalent:
\begin{itemize}
\item[\rm (a)] $T$ is upper frequently hypercyclic;
\item[\rm (b)] $T$ is hypercyclic, and $\text{\emph{UFRec}}(T)$ is a residual set;
\item[\rm (c)] $T$ is hypercyclic, and $\text{\emph{UFRec}}(T)$ is of second category;
\item[\rm (d)] $T$ admits a hypercyclic upper frequently recurrent vector.
\end{itemize}
In that case the set of hypercyclic upper frequently recurrent vectors is residual. Moreover, every hypercyclic upper frequently recurrent vector is upper frequently hypercyclic.
\end{theorem}

\begin{proof} (a) $\Rightarrow$ (b) $\Rightarrow$ (c) $\Rightarrow$ (d). This follows from the fact that the set of upper frequently hypercyclic vectors is either empty or residual, see \cite{BaRu15}, and that the same is true for the set of hypercyclic vectors.

(d) $\Rightarrow$ (a). Let $x$ be a hypercyclic upper frequently recurrent vector, and let $U$ be a non-empty open set. By hypercyclicity of $x$ there is some $m\geq 0$ such that $T^mx \in U$. The continuity of $T$ implies that there is some neighbourhood $V$ of $x$ such $T^m(V)\subset U$. Since $x$ is upper frequently recurrent we have that the set
\[
N(x,V)=\{ n\geq 0: T^n x\in V \}
\]
has positive upper density, and so has $N(x,V)+m$. But
\[
N(x,V)+m\subset\{ n\geq 0: T^nx\in U \}.
\]
This shows that $x$ is upper frequently hypercyclic. This also proves the additional claim.
\end{proof}

However, the analogue of Theorem \ref{reit-hyp}(e) breaks down for upper frequent recurrence, as the following shows.

\begin{example}
Menet \cite{Men17} has constructed a chaotic operator $T$ on $\ell^1(\N)$ that is not upper frequently hypercyclic. Since periodic points are
upper frequently recurrent, $T$ is hypercyclic and upper frequently recurrent without being upper frequently hypercyclic.

Note also that, by Theorem \ref{ufreq-hyp}(c), $\text{UFRec}(T)$ must be of first category. This is in sharp contrast to the fact that the set of upper frequently hypercyclic vectors is always either empty or residual, see \cite{BaRu15}.
\end{example}

In the case of frequent hypercyclicity we have even fewer equivalent conditions. The proof is identical to that of the corresponding part in Theorems \ref{reit-hyp} or \ref{ufreq-hyp}.

\begin{theorem}\label{freq-hyp}
Let $T\in L(X)$. Then the following assertions are equivalent:
\begin{itemize}
\item[\rm (a)] $T$ is frequently hypercyclic;
\item[\rm (b)] $T$ admits a hypercyclic frequently recurrent vector.
\end{itemize}
Moreover, every hypercyclic frequently recurrent vector is frequently hypercyclic.
\end{theorem}

There is a striking difference in hypercyclicity when passing from lower to upper densities: while the set of frequently hypercyclic vectors is always of first category (\cite{Moo13}, \cite{BaRu15}), the set of upper frequently hypercyclic vectors is residual unless empty (\cite{BaRu15}). We have just seen that we lose the latter property for upper frequent recurrence. For frequent recurrence we collect here some cases where $\text{FRec}(T)$ is of first category.

Recall the following notions. The orbit of a vector $x\in X$ for an operator $T\in L(X)$ is called distributionally near to zero (distributionally unbounded) if there is a set $A\subset\N_0$ with $\overline{\text{dens}}(A)=1$ such that $T^nx\to 0$ as $n\to\infty$, $n\in A$ ($p(T^nx)\to\infty$ as $n\to\infty$, $n\in A$, for some continuous seminorm $p(\cdot)$ on $X$, respectively). These two properties, put together, define the notion of a distributionally irregular vector, see \cite{BBMP}.

\begin{theorem}
Let $T\in L(X)$. Suppose one of the following conditions is satisfied:
\begin{itemize}
\item[\rm (a)] $T$ is hypercyclic;
\item[\rm (b)] $T$ has a distributionally unbounded orbit;
\item[\rm (c)] $T$ has a dense set of vectors whose orbits are distributionally near to zero;
\item[\rm (d)] $T$ has a dense set of vectors $x\in X$ such that $T^nx\to 0$ as $n\to\infty$.
\end{itemize}
Then $\text{\emph{FRec}}(T)$ is of first category.
\end{theorem}

\begin{proof}
(a) Let $T$ be hypercyclic. If $\text{FRec}(T)$ is of second category then so is the set of hypercyclic frequently recurrent vectors. But by the previous theorem these are exactly the frequently hypercyclic vectors, which is impossible since these vectors form a set of first category.

(b) By \cite[Proposition 7]{BBMP}, the hypothesis implies that there exists a residual subset of vectors in $X$ with distributionally unbounded orbits. But none of these vectors can be frequently recurrent.

(c) We use the same argument as in (b), based now on \cite[Proposition 9]{BBMP}.

(d) This is a special case of (c).
\end{proof}

Of course, the identity operator tells us that $\text{FRec}(T)$ can be all of $X$. So the following question seems natural.

\begin{question}\label{q2}
Do we always have that either $\text{FRec}(T)=X$ or $\text{FRec}(T)$ is a set of first category?
\end{question}

When we now try to look in the same way at uniformly recurrent vectors then we have gone too far: such vectors can never be hypercyclic. This is obvious for Banach space operators. But it is also valid in general, as follows from a classical result of Furstenberg~\cite[Theorem 1.17]{Fur3}: the closure of the orbit of any uniformly recurrent vector is a minimal set (that is, it does not contain any proper closed invariant subset). Thus, no periodic point can be an accumulation point of the orbit of a uniformly recurrent vector. We give here the proof of this conclusion for the sake of completeness; the argument can also be found in the proof of \cite[Proposition 2]{BMPP16}.

\begin{theorem}\label{urv_not_hc}
Let $T\in L(X)$. Then no periodic point of $T$ is an accumulation point of the orbit of a uniformly recurrent vector. In particular, no uniformly recurrent vector for $T$ is hypercyclic.
\end{theorem}

\begin{proof} Suppose on the contrary that a periodic point $y$ is an accumulation point of the orbit of a uniformly recurrent vector $x$. Then $x$ cannot belong to the (finite) orbit of $y$ under $T$, so that there are disjoint open sets $U$ and $V$ containing $x$ and the orbit of $y$, respectively. Let $m$ be the maximum gap in the return set $N(x,U)$. Then there is a neighbourhood $W$ of $y$ such that $T^j(W)\subset V$ for $j=0,\ldots,m$. By assumption there is some $n\geq 0$ such that $T^nx\in W$. But then $T^{k}x$ belongs to $V$ and therefore not to $U$ for the $m+1$ exponents $k=n,\ldots,n+m$, which is a contradiction.

Since $0$ is a periodic point of every operator, the final conclusion follows.
\end{proof}

Theorem~\ref{urv_not_hc} only leaves the possibility to study hypercyclic operators that \textit{also} have a dense set of uniformly recurrent vectors (or, for that matter, upper frequently recurrent vectors, frequently recurrent vectors). We will not pursue this here.

Let us ask again the following.

\begin{question}\label{q3}
Do we always have that either $\text{URec}(T)=X$ or $\text{URec}(T)$ is a set of first category?
\end{question}

We will get a partial positive answer in Section \ref{s-recpower}. Note that, for periodic points, the corresponding property holds. It is a simple consequence of Baire's theorem that either $\text{Per}(T)$ is of first category or else $T^n=I$ for some $n\geq 1$ (and hence $\text{Per}(T)=X$).

Our next result was motivated by \cite[Corollary 5.20]{GriMaMe}. The authors there show that if an operator $T$ is uniformly recurrent, and if there is a dense of vectors $x\in X$ such that $T^nx\to 0$ as $n\to\infty$, then $T$ is upper frequently hypercyclic. They call this result \textit{somewhat unexpected}. We can give here a more natural (and improved) version of their finding.

\begin{theorem}\label{t-zero_orbits}
Let $T\in L(X)$. Suppose that there is a dense set of vectors $x\in X$ such that $T^nx\to 0$ as $n\to\infty$. Then we have the following:
\begin{itemize}
\item[\emph{(a)}] if $T$ is recurrent then it is hypercyclic;
\item[\emph{(b)}] if $T$ is reiteratively recurrent then it is reiteratively hypercyclic;
\item[\emph{(c)}] if $T$ is upper frequently recurrent then it is upper frequently hypercyclic.
\end{itemize}
\end{theorem}

\begin{proof}
For upper frequent hypercyclicity, it suffices by \cite[Corollary 3.4]{BG} to show that, for any non-empty open set $V$ in $X$ there is some $\delta>0$ such that, for any non-empty open set $U$ in $X$, there is some $x\in U$ such that
\begin{equation}\label{dens}
\overline{\text{dens}}\{ n\geq 0 : T^nx \in V\} >\delta.
\end{equation}

Such a set $V$ contains an upper frequently recurrent vector $v$. Choose open neighbourhood $V_0$ of $v$ and $W$ of zero such that $V_0+W\subset V$. Then the set
\[
A:=\{ n\geq 0 : T^nv \in V_0\}
\]
has positive upper density. Choose $0<\delta<\overline{\text{dens}}(A)$. Now let $U$ be a non-empty open set. By hypothesis there is some $y\in U-v$ such that $T^ny\to 0$ as $n\to\infty$. Then the vector $x:=y+v$ belongs to $U$, and we have that
\[
T^nx = T^ny + T^n v \in W+V_0 \subset V
\]
whenever $n\in A$ is sufficiently large, which implies \eqref{dens}.

The proof for reiterative recurrence and recurrence is similar; see also Theorem \ref{t-zero_orbits2} below.
\end{proof}

The proof, however, breaks down for frequent hypercyclicity.

\begin{question}\label{q4}
Let $T$ be a frequently recurrent operator (or even a chaotic operator) such that $T^nx\to 0$ as $n\to\infty$ for all $x$ from a dense subset of $X$. Does it follow that $T$ is frequently hypercyclic? It seems to be even open whether every chaotic operator with a dense generalized kernel (that is, $\overline{\bigcup_{n\geq 0} \text{ker}(T^n)}=X$) is frequently hypercyclic.
\end{question}

We finish the section with one more natural question. Let $T$ be an invertible operator. Does then a given dynamical property pass from $T$ to its inverse? This is so for hypercyclicity, as is well known, as well as for reiterative hypercyclicity (see \cite {BG}) and recurrence (\cite{CMP}). On the other hand, Menet \cite{Men19a}, \cite{Men19b} has recently shown that the corresponding results are false for (upper) frequent hypercyclicity.

\begin{question}\label{q5}
Let $T$ be an invertible operator. If $T$ is reiteratively recurrent (upper frequently recurrent, frequently recurrent, uniformly recurrent), does $T^{-1}$ have the same property?
\end{question}

\section{Recurrence and power boundedness}\label{s-recpower}

Not surprisingly, power boundedness influences strongly the dynamical properties of an operator. An operator $T\in L(X)$ is called power bounded if the sequence $(T^n)_{n\geq 0}$ is equicontinuous, that is, if for any $0$-neighbourhood $W_1$ there is a $0$-neighbourhood $W_2$ such that, for any $n\geq 0$,
\[
T^n(W_2)\subset W_1;
\]
by the Banach-Steinhaus theorem, this is equivalent to saying that every orbit under $T$ is bounded, see \cite{Rud74}.

The following is then obvious; see also \cite[Lemma 3.1]{CMP}.

\begin{theorem}\label{t-power}
Let $T\in L(X)$. If $T$ is power bounded, then the sets $\text{\emph{URec}}(T)$, $\text{\emph{FRec}}(T)$, $\text{\emph{UFRec}}(T)$, $\text{\emph{RRec}}(T)$ and $\text{\emph{Rec}}(T)$ are closed.
\end{theorem}

\begin{proof} We only consider uniform recurrence. Let $x\in \overline{\text{URec}(T)}$ and $W$ be a $0$-neighbourhood. Choose a $0$-neighbourhood $W_1$ such that $W_1+W_1+W_1\subset W$. By power boundedness, there is a $0$-neighbourhood $W_2\subset W_1$ such that $T^n(W_2)\subset W_1$ for all $n\geq 1$.

Now, by assumption, there is some $y\in \text{URec}(T)$ such that $x-y, y-x\in W_2$. In addition, the set $A:=\{n\geq 0: T^ny-y\in W_1\}$ is syndetic. But then we have for $n\in A$ that
\[
T^n x - x = T^n (x-y) +T^ny-y + y - x  \in W_1 +W_1 +W_1\subset W.
\]
Since $W$ is arbitrary, $x$ is uniformly recurrent.
\end{proof}

This shows that, for every power bounded operator, recurrence of the operator implies that every vector is recurrent; and similarly for the other notions of recurrence.

On the other hand, for an operator acting on a Banach space, every uniformly recurrent vector has a bounded orbit. Thus we immediately obtain
the following partial answer to Question \ref{q3}.

\begin{corollary}\label{c-powerUR}
Let $X$ be a Banach space, and $T\in L(X)$ be a uniformly recurrent operator. Then either $\text{\emph{URec}}(T)$ is of first category, or $\text{\emph{URec}}(T)=X$.
\end{corollary}

\begin{proof}
Suppose that the set $\text{URec}(T)$ is of second category. Then so is the set of vectors with bounded orbit under $T$, which by the Banach-Steinhaus theorem implies that $T$ is power bounded, see \cite{Rud74}. By the previous theorem, $\text{URec}(T)$ is then closed, and also dense by hypothesis, so $\text{URec}(T)=X$.
\end{proof}

Now, for Fr\'echet space operators, a uniformly recurrent orbit is not necessarily bounded, so that one cannot argue as in the proof of Corollary \ref{c-powerUR}. An example is given by the backward shift on the space $\mathbb{K}^\mathbb{N}$ of all (real or complex) sequences, see \cite[Example~1]{GMOP}. We give here an example on a Fr\'echet space with a continuous norm. The type of operator considered in this example might also be of independent interest.

\begin{example}\label{ex_urec_unbounded}
Let $X$ be the space of doubly indexed sequences $x=(x_{k,j})_{k\geq 0, 0\leq j < 2^{k}}$ such that, for any $n\geq 1$,
\[
p_n(x):= \sum_{k=0}^\infty \frac{1}{2^k} \max_{0\leq j < 2^k} |x_{k,j}| + \sum_{k=2}^\infty k \max_{\substack{1\leq m \leq n\\m<2^{k-1}}} |x_{k,2^{k-1}+m}|<\infty.
\]
Figure \ref{f-indices} indicates the area of indices that is involved in the second sum. When endowed with the increasing sequence of (semi)norms $p_n$, $n\geq 1$, $X$ obviously becomes a separable Fr\'echet space.

\begin{figure}
\caption{Indices for the seminorms $p_n$}\label{f-indices}%
\centering
\begin{tikzpicture}
\draw[->] (0,0)- -(5,0) node[right] {$j$};
\draw[->] (0,0)- -(0,-5) node[below] {$k$};
\draw[domain=0:5] plot (\x,{-1.5*ln (\x+1)});
\draw[domain=0:5] plot (\x,{-3*ln (\x+1)});
\draw[domain=0.1916:5] plot (\x,{-3*ln (\x+.9)});
\draw[domain=0.866:5] plot (\x,{-3*ln (\x+.5)});
\fill[lightgray]  plot[domain=5:0.1916] (\x,{-3*ln (\x+.9)})-- plot[domain=0.1916:0.8686] (\x,{-1.5*ln (\x+1)})-- plot[domain=0.866:5] (\x,{-3*ln (\x+.5)});
\node[above right] at (5,-3) {\small $2^k$};
\node[above right] at (2.6,-5) {\small $2^{k-1}$};
\end{tikzpicture}
\end{figure}

We consider the operator $T$ on $X$ given by
\[
T(x_{k,j})_{k,j} = (x_{k, j+1 (\text{mod } 2^k)})_{k,j},
\]
that is, a simple row-wise rotation. To see that $T$ is continuous, fix $n\geq 1$. Choose $l\geq 2$ so that $2^l\geq 2(n+2)$, which implies that $n +1 < 2^{k-1}$ for all $k\geq l$. Then we have for $x\in X$ that
\begin{align*}
p_n(Tx) &= \sum_{k=0}^\infty \frac{1}{2^k} \max_{0\leq j < 2^k} |[Tx]_{k,j}| + \sum_{k=2}^\infty k \max_{\substack{1\leq m \leq n\\m<2^{k-1}}} |[Tx]_{k,2^{k-1}+m}|\\
&=\sum_{k=0}^\infty \frac{1}{2^k} \max_{0\leq j < 2^k} |x_{k,j}| + \sum_{k=2}^{l-1} k \max_{\substack{1\leq m \leq n\\m<2^{k-1}}} |[Tx]_{k,2^{k-1}+m}|+\sum_{k=l}^\infty k \max_{1\leq m \leq n} |x_{k,2^{k-1}+m+1}|\\
&\leq \sum_{k=0}^\infty \frac{1}{2^k} \max_{0\leq j < 2^k} |x_{k,j}| + (l-1)2^{l-1}\sum_{k=2}^{l-1} \frac{1}{2^k} \max_{0\leq j < 2^k} |x_{k,j}|+\sum_{k=l}^\infty k \max_{1\leq m \leq n+1} |x_{k,2^{k-1}+m}|\\
&\leq (1+(l-1)2^{l-1})p_{n+1}(x),
\end{align*}
which proves continuity.

Now, consider the vector $x=(x_{j,k})\in X$ given by $x_{k,0}=1$ for $k\geq 0$ and all other $x_{k,j}=0$. Then $x$ is uniformly recurrent for $T$. Indeed, let $n\geq 1$ and $\varepsilon>0$. Choose $l\geq 0$ such that $2^l>\max(n,1/\varepsilon)$. Let $\nu\geq 0$. First we observe that
\[
[T^{\nu 2^l}x]_{k,j} = x_{k,j},\quad k\leq l,  0\leq j < 2^k.
\]
On the other hand, for $k>l$, the fact that $x_{k,j}=0$ for all $j\neq 0$ implies that $[T^{\nu 2^l}x]_{k,j} = 0$ whenever $j$ is not a multiple of $2^{l}$. Now since, for these $k$, $2^{k-1}$ is a multiple of $2^l$ and $n<2^l$, we have that
\[
[T^{\nu 2^l}x]_{k,2^{k-1}+m} = 0,\quad k > l,  1\leq m\leq n;
\]
note that $m<2^{k-1}$ is automatic. Thus we have for any $\nu\geq 0$
\[
p_n(T^{\nu l}x-x) = \sum_{k>l} \frac{1}{2^k}=\frac{1}{2^l}<\varepsilon.
\]
This shows that $x$ is uniformly recurrent.

On the other hand, by construction, the orbit of $x$ is unbounded. It suffices to observe that for $k\geq 2$
\[
[T^{2^{k-1}+1}x]_{k,2^{k-1}+1}=1,
\]
so that $p_1(T^{2^{k-1}+1}x)\geq k$.
\end{example}

The vector $x$ considered above is not periodic, but it enjoys the following property: for any neighbourhood $U$ of $x$ there is some $k\geq 1$ such that $T^{nk}x\in U$ for all $n\geq 0$. Such points have been called regularly recurrent (or regularly almost periodic) in non-linear dynamics, see \cite{GoHe}, \cite{BK04}.

The examples show that, for Fr\'echet spaces, the proof for Corollary \ref{c-powerUR} breaks down at a very early stage. One may wonder what kind of (weak) boundedness the orbit of a uniformly recurrent vector possesses in the setting of Fr\'echet spaces.

On the other hand, for power bounded operators we have a strong form of boundedness.

\begin{theorem}\label{t-power2}
Let $T\in L(X)$ be power bounded. If $x$ is a uniformly recurrent vector for $T$ then the closure of its orbit is compact.
\end{theorem}

\begin{proof}
We show, equivalently, that the orbit of $x$ is totally bounded, that is, for any 0-neighbourhood $W$ there are finitely many points $x_0,\ldots,x_N$ such that the orbit is contained in $\bigcup_{n=0}^N (x_n+W)$. Thus let $W$ be a 0-neighbourhood. By power boundedness there is a 0-neighbourhood $W_0$ such that $T^n(W_0)\subset W$ for all $n\geq 0$. Let $N$ be the maximum gap in the return set $N(x,x+W_0)$. Then we have that
\[
\{T^k x : k\geq 0\}\subset \bigcup_{n=0}^N T^n(x+W_0)\subset \bigcup_{n=0}^N (T^nx+W),
\]
which implies the claim.
\end{proof}

We already recalled Furstenberg's result which says that the closure of the orbit of a uniformly recurrent vector is a minimal set. The dynamics on minimal compact sets (like irrational rotations on the torus) is a matter of study in non-linear dynamics.

\section{Recurrence, the unit circle, and the spectrum}\label{s-recspec}

In this section we study some properties of recurrence whose hypercyclic analogues belong to the fundamental results in linear dynamics. Costakis et al.~\cite{CMP} have obtained the following: for any recurrent operator $T$,
\begin{itemize}
\item[--] $T^p$ is recurrent for any $p\geq 1$; in fact, $\text{Rec}(T^p)=\text{Rec}(T)$;
\item[--] $\lambda T$ is recurrent whenever $|\lambda|=1$; in fact, $\text{Rec}(\lambda T)=\text{Rec}(T)$;
\end{itemize}
moreover, if $X$ is a complex Banach space, then
\begin{itemize}
\item[--] every component of the spectrum $\sigma(T)$ meets the unit circle;
\item[--] the point spectrum $\sigma_p(T^*)$ of its adjoint $T^*$ is contained in the unit circle.
\end{itemize}

We start by looking at the first two properties for other notions of recurrence. Our approach uses in a crucial way an idea of Bayart and Matheron \cite[Section 6.3.3]{BaMa09}. Let us say that a family $\mathcal{F}$ of subsets of $\N_0$ has the \textit{\underline{cu}t-\underline{s}hift-and-\underline{p}aste} property, CuSP for short, if for any $I_1,\ldots,I_q\subset \N_0$ with $\N_0=\bigcup_{j=1}^q I_j$ and $n_1,\ldots,n_q\in\mathbb{N}_0$ ($q\geq 1$),
\[
A\in \mathcal{F}\ \Rightarrow \ \bigcup_{j=1}^q (n_j + A\cap I_j) \in \mathcal{F}.
\]
Then \cite[Lemma 6.29]{BaMa09} says that the family of sets of positive lower density has CuSP.

\begin{lemma}\label{l-chip}
The following families have CuSP: The syndetic subsets and the infinite subsets of $\N_0$, and the sets of positive lower density, positive upper density, and positive upper Banach density.
\end{lemma}

\begin{proof}
The case of positive lower density is proved in \cite[Lemma 6.29]{BaMa09}; the same proof also covers the case of positive upper density. The result is obvious for the family of infinite subsets.

For the remaining cases, fix $I_1,\ldots,I_q\subset \N_0$ with $\N_0=\bigcup_{j=1}^q I_j$ and $n_1,\ldots,n_q\in\mathbb{N}_0$, where $q\geq 1$. Set $M=\max(n_1,\ldots,n_q)$.

Let $A\subset \N_0$ be such that $\overline{\text{Bd}}(A)>\delta>0$. Fix $N_0\geq 0$. Then there is some $N_1\geq \max(N_0,M)$ such that, for every $N\geq N_1$, there is some $m\geq 0$ such that
\[
\frac{1}{N+1} \text{card}(A\cap [m,m+N]) > \delta.
\]
Now, for some $k$, $1\leq k\leq q$, we have that
\[
\text{card} (A\cap I_k\cap [m,m+N]) \geq \frac{1}{q}\text{card} (A\cap [m,m+N]);
\]
moreover,
\[
\text{card}\big( (n_k+A\cap I_k)\cap [m,m+N+M]\big) \geq \text{card} (A\cap I_k\cap [m,m+N]).
\]
Altogether we obtain that
\begin{align*}
\frac{1}{N+M+1}\text{card} &\Big(\Big(\bigcup_{j=1}^q (n_j+A\cap I_j)\Big)\cap [m,m+N+M]\Big) \geq\\
  &\frac{1}{q} \frac{N+1}{N+M+1}\frac{1}{N+1} \text{card}(A\cap [m,m+N])\geq \frac{1}{2q} \delta,
\end{align*}
which shows that $\bigcup_{j=1}^q (n_j + A\cap I_j)$ has positive upper Banach density.

Finally, let $A\subset \N_0$ be a syndetic set. Then there is some $N\geq 1$ such that every integer interval of length $N$ contains some element of $A$. Let $J=[m_1,m_2]$ be an integer interval of length $N+M$. Then the interval $[m_1,m_1+N]$ contains an element $m\in A$. By assumption there is some $k$, $1\leq k\leq q$, such that $m\in A\cap I_k$. But then $m+n_k$ is an element of $\bigcup_{j=1}^q (n_j + A\cap I_j)$ that belongs to $J$. Thus the set $\bigcup_{j=1}^q (n_j + A\cap I_j)$ is syndetic.
\end{proof}

As a first application we deduce an Ansari-type result for various forms of recurrence. Recall that Ansari \cite{Ans95} has proved that, for any $p\geq 1$, $T$ and $T^p$ have the same hypercyclic vectors. Her proof uses in an essential way a connectedness argument; for recurrence, the argument is simpler.

\begin{theorem}\label{t-ansari}
Let $T\in L(X)$, and let $p\geq 1$. Then $T$ and $T^p$ have the same uniformly recurrent (frequently recurrent, upper frequently recurrent, reiteratively recurrent) vectors.

In particular, if $T$ is uniformly recurrent (frequently recurrent, upper frequently recurrent, reiteratively recurrent) then so is $T^p$.
\end{theorem}

\begin{proof} Let $p\geq 1$ be given. We will show that $\text{URec}(T) = \text{URec}(T^p)$, where we only use two properties of the family of syndetic sets: CuSP and the fact that $A\subset \N_0$ is syndetic if and only if $pA=\{pn : n\in A\}$ is. Therefore the remaining assertions can be proved in exactly the same way.

It suffices to show that $\text{URec}(T) \subset \text{URec}(T^p)$, the converse inclusion being trivial. We may also suppose that $p$ is a prime number. Thus, let $x$ be a uniformly recurrent vector for $T$. Let $(U_k)_{k\geq 1}$ be a decreasing sequence of neighbourhoods of $x$ that forms a local base. For $k\geq 1$, we define
\[
J_k=\{ j\in \{0,\ldots,p-1\} : \text{there exists $n\geq 0$ with $n=j (\text{mod } p)$ and $T^nx\in U_k$}\}.
\]
Then $(J_k)_{k\geq 1}$ is a decreasing sequence of non-empty finite sets, which therefore stabilizes. That is, there is a non-empty set $J\subset\{0,\ldots,p-1\}$ and some $k_0\geq 1$ such that $J_k=J$ for all $k\geq k_0$.

We claim that $J$ is a subgroup of $\Z/p\Z$. Indeed, let $j,j'\in J$. First, there is some $n\geq 0$ with $n=j (\text{mod } p)$
such that $T^nx\in U_{k_0}$. By continuity there is some $l\geq k_0$ such that $T^n(U_l)\subset U_{k_0}$. Now, since $j'\in J_l$, there is then some $n'\geq 0$ with $n'=j' (\text{mod } p)$ such that $T^{n'}x\in U_l$. Altogether we have that $T^{n+n'}x=T^n(T^{n'}x)\subset U_{k_0}$, hence, by definition, $j+j' (\text{mod } p) \in J_{k_0}=J$.

Since $p$ is prime, $\Z/p\Z$ only has two subgroups. We distinguish these two cases.

(a) We first assume that $J=\{0\}$. Then the sets
\[
A_k:= \{n\geq 0 : T^nx\in U_k\}, \quad k\geq k_0
\]
only consist of multiples of $p$, and they are syndetic by hypothesis. Thus the sets $\frac{1}{p}A_k$ are syndetic, and $(T^p)^n x\in U_k$ for all $n\in \frac{1}{p}A_k$. This shows that $x$ is uniformly recurrent for $T^p$.

(b) Now assume that $J=\{0,\ldots,p-1\}$, hence $J_k=\{0,\ldots,p-1\}$ for all $k\geq 1$. Let $k\geq 1$. For any $j\in J$ we can find some $n_j=p-j(\text{mod } p)$ such that $T^{n_j}x\in U_k$. By continuity there is some $l\geq 1$ such that, for any $j\in J$,
\[
T^{n_j}(U_l)\subset U_k.
\]
By our hypothesis, the set
\[
A_l=\{ n\geq 0 : T^nx \in U_l\}
\]
is syndetic. In order to apply Lemma \ref{l-chip}, we set
\[
I_j = \{ n\geq 0 : n=j (\text{mod } p)\}, \quad j\in J.
\]
Let $n\in A_l\cap I_j$, $j\in J$. Then we have that
\[
T^{n_{j}+n} x = T^{n_{j}}(T^nx)\in U_k.
\]
In other words,
\[
A:=\bigcup_{j\in J} (n_{j} + A_l\cap I_j) \subset \{ n\geq 0 : T^nx\in U_k\}
\]
(we only need the fact that $J$ is the full subgroup for the existence of some $n_{j}$ for any $j\in J$). By Lemma \ref{l-chip}, $A$ is a syndetic set. Moreover, if $m\in A$ then there are $j\in J$ and $n\geq 0$, $n=j (\text{mod } p)$ such that
\[
m=n_{j} + n = p-j+j(\text{mod } p)= 0 (\text{mod } p).
\]
Thus the set $\frac{1}{p}A$ is syndetic and $(T^p)^n x\in U_k$ for all $n\in \frac{1}{p}A$. Since $k\geq 1$ was arbitrary, we see that $x$ is uniformly recurrent for $T^p$.
\end{proof}

\begin{remark}\label{r-tvs1}
In fact, Theorem \ref{t-ansari} holds for any continuous map on any topological space, and in particular for every operator on any topological vector space. In the proof one only needs to replace the countable local base $(U_k)_{k\geq 1}$ by the filter of all neighbourhoods at $x$. See also \cite[Remark 2.4]{CMP}.
\end{remark}

As usual, the $\lambda T$-problem is closely related to the $T^p$-problem: Le\'on and M\"uller \cite{LeMu04} have shown that, for any scalar $\lambda$ of modulus 1, $T$ and $\lambda T$ have the same hypercyclic vectors. Like for hypercyclicity, the proof in the $\lambda T$-case for recurrence requires somewhat more work than in the $T^p$-case.

\begin{theorem}\label{t-LeMu}
Let $T\in L(X)$, and let $\lambda$ be a scalar with $|\lambda|=1$. Then $T$ and $\lambda T$ have the same uniformly recurrent (frequently recurrent, upper frequently recurrent, reiteratively recurrent) vectors.

In particular, if $T$ is uniformly recurrent (frequently recurrent, upper frequently recurrent, reiteratively recurrent) then so is $\lambda T$.
\end{theorem}

\begin{proof} This time we only use the CuSP property, so it again suffices to do the uniformly recurrent case.

The real scalar case already follows from Theorem \ref{t-ansari} because $(-T)^2 =T^2$. Thus we need only consider complex scalars. Alternatively one can also do the following proof for $\R$ instead of $\C$.

It obviously suffices to show that $\text{URec}(T)\subset \text{URec}(\lambda T)$ whenever $|\lambda|=1$. Thus, let $\lambda\in\C$ with $|\lambda|=1$, and let $x$ be a uniformly recurrent vector for $T$. Let $(U_k)_{k\geq 1}$ be a decreasing sequence of neighbourhoods of $x$ that forms a local base. For $k\geq 1$, we define
\[
\Lambda_k=\{ \mu\in\T: \text{there exists $n\geq 0$ with $\lambda^n=\mu$ and $T^nx\in U_k$}\},
\]
where $\T=\{z\in\C : |z|=1\}$ denotes the unit circle. Then $(\Lambda_k)_{k\geq 1}$ is a decreasing sequence of non-empty subsets of $\T$. Let
\[
\Lambda = \bigcap_{k=1}^\infty \overline{\Lambda_k}.
\]
Since $\Lambda$ is the intersection of a decreasing sequence of non-empty closed sets, it is a non-empty closed subset of $\T$.

We now claim that $\Lambda$ is a subsemigroup of the multiplicative group $\T$. To see this, let $\mu,\mu'\in \Lambda$. Let $k\geq 1$ and $\varepsilon>0$. Then there is some $\mu_k\in\Lambda_k$ such that $|\mu-\mu_k|<\varepsilon$. This implies that there is some $n_k\geq 0$ such that $\lambda^{n_k}=\mu_k$ and $T^{n_k}x\in U_k$. By continuity there is some $l\geq 1$ such that $T^{n_k}(U_l)\subset U_k$. We then find some $\mu'_l\in\Lambda_l$ such that $|\mu'-\mu'_l|<\varepsilon$ and hence some $n'_l\geq 0$ such that $\lambda^{n'_l}=\mu'_l$ and $T^{n'_l}x\in U_l$. Altogether we get that $T^{n_k+n'_l}x \in T^{n_k}(U_l)\subset U_k$. Since $\lambda^{n_k+n'_l}=\mu_k\mu'_l$, we deduce that $\mu_k\mu'_l\in \Lambda_k$. On the other hand,
\[
|\mu\mu'-\mu_k\mu'_l |\leq |\mu-\mu_k|\ |\mu'| + |\mu_k|\ |\mu'-\mu'_l|<2\varepsilon.
\]
Since $\varepsilon>0$ is arbitrary, $\mu\mu'\in\overline{\Lambda_k}$, and that for any $k\geq 1$. Thus $\mu\mu'\in \Lambda$, as had to be shown.

As a consequence, there are only two possibilities for $\Lambda$, see \cite[pp. 170--171]{GrPe11}.

(a) This time it is easier to start with the full group: $\Lambda = \T$. Let $U$ be a neighbourhood of $x$. By continuity of scalar multiplication there is some $\varepsilon>0$ and some $k\geq 1$ such that $B(1,\varepsilon)U_k\subset U$, where $B(z_0,\varepsilon)=\{z\in \C : |z-z_0|<\varepsilon\}$. Since $\Lambda=\T$, the set $\Lambda_k$ is dense in $\T$. Using compactness there are $N\geq 1$ and $n_j\geq 0$ with $T^{n_j}x\in U_k$, $j=1,\ldots,N$, such that
\begin{equation}\label{cover}
\T\subset \bigcup_{j=1}^N B(\lambda^{n_j},\varepsilon).
\end{equation}
By continuity there exists $l\geq 1$ such that, for $j=1,\ldots,N$,
\[
T^{n_j}(U_l)\in U_k,
\]
and we have that the set
\[
A_l:=\{n\geq 0 : T^n x \in U_l\}
\]
is syndetic. Also, it follows from \eqref{cover} that the sets
\[
I_j := \{ n\geq 0 : \lambda^{n_j+n} \in B(1,\varepsilon)\},\quad j=1,\ldots,N,
\]
form a cover of $\N_0$.

Now, if $n\in A_l\cap I_j$, $j=1,\ldots,N$, then
\[
(\lambda T)^{n_j+n}x = \lambda^{n_j+n} T^{n_j}(T^nx) \in B(1,\varepsilon)T^{n_j}(U_l)\subset B(1,\varepsilon)U_k \subset U.
\]
This shows that
\[
A:=\bigcup_{j=1}^N (n_j+A_l\cap I_j) \subset \{n\geq 0 : (\lambda T)^nx\in U\},
\]
and it follows from Lemma \ref{l-chip} that $\{n\geq 0 : (\lambda T)^nx\in U\}$ is syndetic. Thus $x$ is uniformly recurrent for $\lambda T$.

(b) It remains the case when there is some $N\geq 1$ such that $\Lambda = \{ \text{e}^{2\pi i\frac{j}{N}} : j=1,\ldots, N\}$. Let $U$ be a neighbourhood of $x$, and then $\varepsilon>0$ and $k'\geq 1$ such that $B(1,\varepsilon)U_{k'}\subset U$. It follows from a simple compactness argument that there is some $k\geq k'$ such that
\[
\Lambda_k\subset \bigcup_{j=1}^N B\big(\text{e}^{2\pi i\frac{j}{N}},\tfrac{\varepsilon}{2}\big).
\]
Moreover, since $\text{e}^{2\pi i\frac{-j}{N}}\in\Lambda\subset \overline{\Lambda_k}$, $j=1,\ldots,N$, there are $n_j\geq 0$ with $T^{n_j}x\in U_k$ and
\[
\big|\lambda^{n_j} - \text{e}^{2\pi i\frac{-j}{N}}\big|<\tfrac{\varepsilon}{2}.
\]
As before there is some $l\geq k$ such that, for $j=1,\ldots,N$,
\[
T^{n_j}(U_l)\subset U_k,
\]
and the set
\[
A_l=\{n\geq 0: T^nx\in U_l\}
\]
is syndetic.

Now, let $n\in A_l$. Then $\lambda^n\in\Lambda_l\subset \Lambda_k$, so that there is some $j\in \{1,\ldots,N\}$ such that $|\lambda^n-\text{e}^{2\pi i\frac{j}{N}}|<\tfrac{\varepsilon}{2}$, hence
\[
|\lambda^{n_{j}}\lambda^{n} -1|\leq \big|\lambda^{n_{j}}- \text{e}^{2\pi i\frac{-j}{N}}\big|  |\lambda^n| + \big|\text{e}^{2\pi i\frac{-j}{N}}\lambda^n-1\big| <\varepsilon.
\]
This shows that the sets
\[
I_j := \{ n\geq 0 : \lambda^{n_j+n} \in B(1,\varepsilon)\},\quad j=1,\ldots,N,
\]
cover $A_l$. From here the proof can be finished as in case (a). In order that the sets $I_j$ cover $\N_0$ one might add $I_0=\N_0\setminus A_l$, which has no influence in the sequel.
\end{proof}

\begin{remark}\label{r-tvs2}
Again, by considering the neighbourhood filter instead of the countable local base $(U_k)_{k\geq 1}$, the same proof shows that Theorem \ref{t-LeMu} holds for any operator on any topological vector space.
\end{remark}

Our proofs for the $T^p$- and $\lambda T$-problems work equally well for recurrent operators and therefore provide alternative, if longer, proofs to those of Costakis et al.~\cite{CMP}.

If we combine the last two theorems with Theorems \ref{ufreq-hyp} and \ref{freq-hyp} we obtain new proofs of Ansari-type and Le\'on-M\"uller type theorems for upper frequent hypercyclicity and frequent hypercyclicity that are due to Shkarin \cite{Shk09} and Bayart, Grivaux and Matheron \cite{BaGr06}, \cite{BaMa09}, respectively.

The corresponding results for reiterative hypercyclicity follow with Theorem \ref{reit-hyp}. They seem to be new. Note, however, that the only issue is that reiterative hypercyclicity passes to $T^p$ or $\lambda T$ because, for a reiteratively hypercyclic operator, the sets of reiteratively hypercyclic and hypercyclic vectors coincide, see \cite{BMPP16}.

\begin{corollary}
Let $T\in L(X)$ be reiteratively hypercyclic.

\emph{(a)} If $p\geq 1$, then $T^p$ is reiteratively hypercyclic.

\emph{(b)} If $\lambda$ is a scalar with $|\lambda|=1$, then $\lambda T$ is reiteratively hypercyclic.
\end{corollary}

We have another interesting application of Theorem \ref{t-LeMu}. In \cite[Question 7.11]{GriMaMe}, the authors ask whether any Banach space operator with a dense set of uniformly recurrent vectors must have a non-zero periodic point. Since $\lambda I$ with $\lambda \in\mathbb{T}$ not a root of unity provides a counter-example, see also Remark \ref{r-GMM} below, the authors probably were only interested in hypercyclic operators. Still, a negative answer follows form the theorem above and an important counter-example of Bayart and Berm\'udez \cite{BaBe09}.

\begin{corollary}\label{c-unifnotper}
There exists a hypercyclic operator on Hilbert space that has a dense set of uniformly recurrent vectors but no non-zero periodic points.
\end{corollary}

\begin{proof}
In \cite[Theorem 3.1]{BaBe09} it is proved that there exists a chaotic operator $T$ on complex Hilbert space such that $\lambda T$ is not chaotic for some $\lambda\in\T$. Indeed, the proof even shows that the point spectrum of $\lambda T$ contains no root of unity, so that $\lambda T$ has no non-zero periodic points, see \cite[Proposition 2.33]{GrPe11}. Now since periodic points are uniformly recurrent vectors, the operator $T$ is uniformly recurrent, and then so is $\lambda T$ by Theorem \ref{t-LeMu}.
\end{proof}

\begin{remark}\label{r-GMM}
Let us mention that the corollary can be proved without Theorem \ref{t-LeMu}. Indeed, if $x$ is a periodic point for an operator $T$, then it follows rather directly that $x$ is uniformly recurrent for $\lambda T$ for any $\lambda \in \T$. To see this, suppose that $T^{N}x=x$ for some $N\geq 1$, and let $\lambda \in \T$. Let $\varepsilon>0$. It is well known (see also Lemma \ref{l-kron} below) that there is then a syndetic set $A\subset \N_0$ such that $|(\lambda^{N})^n-1|<\frac{\varepsilon}{\|x\|}$ for all $n\in A$; of course we may assume that $x\neq 0$. Hence $|(\lambda T)^{nN}x-x|=|(\lambda^N)^n-1|\|x\|< \varepsilon$ for all $n\in A$. This shows that $x$ is uniformly recurrent for $\lambda T$.
\end{remark}

We turn to the spectrum of recurrent operators if the underlying space is a complex Banach space. By Costakis et al.~\cite{CMP} we know that every component meets the unit circle. We have additional information when $T$ is upper frequently recurrent. Shkarin \cite[Theorem 1.2 and its proof]{Shk09} showed that the spectrum of an upper frequently hypercyclic operator cannot have isolated points. His argument also serves to show the following.

\begin{theorem}
Let $X$ be a complex Banach space and $T\in L(X)$ be an upper frequently recurrent operator.
\begin{itemize}
\item[\rm (a)] If $\sigma(T)=\{\lambda\}$ is a singleton, then $|\lambda|=1$ and $T=\lambda I$.
\item[\rm (b)] If $\sigma(T)$ has an isolated point $\lambda\in\C$, then $|\lambda|=1$ and there are non-trivial $T$-invariant closed subspaces $M_1$ and $M_2$ of $X$ such that $X=M_1\oplus M_2$ and $T|_{M_1} = \lambda I|_{M_1}$.
\end{itemize}
\end{theorem}

In particular, in either case, $T$ cannot also be hypercyclic; indeed, in (b), the usual quasi-conjugacy argument (see \cite[Proposition 2.25]{GrPe11}) would imply that $T|_{M_1} = \lambda I|_{M_1}$ was hypercyclic, which is absurd. Note that this result also contains the well known fact that the spectrum of any chaotic operator has no isolated points, see \cite{BMP01}, \cite[Proposition 6.37]{BaMa09}.

\begin{proof} (a) By the result of Costakis et al.~\cite{CMP} mentioned above we have that $|\lambda|=1$. Let $S=\lambda^{-1} T$, which is also upper frequently recurrent by Theorem \ref{t-LeMu}. On the other hand, an analysis of Shkarin's argument, see \cite[Remark on p. 153]{BaMa09}, shows that if $S\neq I$ then one can find a non-empty open set $U$ of $X$ such that $\{n\geq 0 : S^nx\in U\}$ has upper density zero for all $x\in X$; in particular, no vector in $U$ is upper frequently recurrent for $S$. Thus $S=I$, hence $T=\lambda I$.

(b) This follows from (a) by the usual argument employing the Riesz decomposition theorem and quasi-conjugacy, see for example the proofs of \cite[Proposition 6.37]{BaMa09} or \cite[Proposition 5.7]{GrPe11}.
\end{proof}

It is not clear whether the result extends to reiterative recurrence. By a result of Salas \cite{Sal95}, see also \cite[Example 8.4]{GrPe11}, there exist hypercyclic compact perturbations $T=I+K$ of the identity with $\sigma(T)=\{1\}$.

\begin{question}\footnote{This question has recently been solved in the negative; see \cite{CarMur2}.}\label{q8}
(a) Can the spectrum of a reiteratively hypercyclic operator be a singleton?

(b) Does there exist a reiteratively hypercyclic compact perturbation of the identity?
\end{question}

\section {Recurrence of weighted backward shift operators}\label{s-recshift}

Backward shifts are the best understood class of operators in linear dynamics. In particular they will serve us here to distinguish five of the six types of recurrence considered in \eqref{eq-incl}.

Apart from the proof of the latter fact, this section contains no proofs: the other results are special cases of stronger results proved either in Section \ref{s-recsize}, or in a forthcoming paper by the first two authors \cite{BG2}, or by other authors. We find it nonetheless instructive to highlight the recurrence behaviour of weighted shifts.

We first need some terminology. A Fr\'echet sequence space (over $\N$) is a Fr\'echet space that is a subspace of the space $\mathbb{K}^{\mathbb{N}}$ of all (real or complex) sequences and such that each coordinate functional $x=(x_n)_{n\geq 1}\to x_k$, $k\geq 1$, is continuous. The canonical unit sequences are denoted by $e_k=(\delta_{k,n})_{n\geq 1}$. A weight sequence is a sequence $w=(w_n)_{n\geq 1}$ of non-zero scalars. The (unilateral) weighted backward shift $B_w$ is then defined by $B_w(x_n)_{n\geq 1}=(w_{n+1}x_{n+1})_{n\geq 1}$.

Fr\'echet sequence spaces over $\Z$ and bilateral weighted backward shifts are defined analogously.

Now, Theorem \ref{t-zero_orbits} applies in particular to unilateral weighted backward shifts.

\begin{corollary}\label{c-unishift}
Let $X$ be a Fr\'echet sequence space over $\N$ in which $(e_n)_{n\geq 1}$ is a basis. Suppose that the weighted backward shift $B_w$ is an operator on $X$. Then we have the following:
\begin{itemize}
\item[\emph{(a)}] if $B_w$ is recurrent then it is hypercyclic;
\item[\emph{(b)}] if $B_w$ is reiteratively recurrent then it is reiteratively hypercyclic;
\item[\emph{(c)}] if $B_w$ is upper frequently recurrent then it is upper frequently hypercyclic.
\end{itemize}
\end{corollary}

Note that each of the hypercyclic properties in the corollary have been characterized in terms of the weights, at least if the basis is unconditional, see \cite[Theorem 4.8]{GrPe11}, \cite[Theorem 5.1]{BG}.

There is a considerable strengthening of assertion (a). By a remarkable result of Chan and Seceleanu \cite{ChSe}, if a weighted shift on $\ell^p(\N)$, $1\leq p<\infty$, admits an orbit with a non-zero limit point, then it is already hypercyclic. Recently, this was extended by He et al.~\cite[Lemma 5]{HHY18} to all Fr\'echet sequence spaces over $\N$ in which $(e_n)_{n\geq 1}$ is an unconditional basis; see also \cite{BG2}.

For bilateral weighted shifts we have the analogue of (a).

\begin{theorem}\label{t-bilshift}
Let $X$ be a Fr\'echet sequence space over $\Z$ in which $(e_n)_{n\in \Z}$ is a basis. Suppose that the weighted backward shift $B_w$ is an operator on $X$. If $B_w$ is recurrent then it is hypercyclic.
\end{theorem}

This was proved by Costakis and Parissis \cite[Proposition 5.1]{CP} for the space $\ell^2(\Z)$. The general case can be shown by combining the proof of these authors with the one of \cite[Theorem 4.12(a)]{GrPe11}, adding a standard conjugacy argument. But, again, Chan and Seceleanu \cite{ChSe} have the stronger result that if a weighted shift on $\ell^p(\Z)$, $1\leq p<\infty$, admits an orbit with a non-zero limit point, then it is hypercyclic. In \cite{BG2}, this is extended to all Fr\'echet sequence spaces over $\Z$ in which $(e_n)_{n\in \Z}$ is an unconditional basis.

Several questions remain (see also Question \ref{q4}).

\begin{question}\label{q9}
Does the analogue of Corollary \ref{c-unishift} hold for frequent recurrence? Does the analogue of Theorem \ref{t-bilshift} hold for reiterative (upper frequent, frequent) recurrence?
\end{question}

Moreover, the work by Chan and Seceleanu might suggest that the existence of a single non-zero vector with some recurrence property implies some type of hypercyclicity. This is indeed the case for uniform recurrence.

\begin{theorem}\label{t-rechc}
Let $X$ be a Fr\'echet sequence space (over $\N$ or $\Z$) in which $(e_n)_n$ is an unconditional basis. Suppose that the (unilateral or bilateral) weighted backward shift $B_w$ is an operator on $X$. If $B_w$ admits a non-zero uniformly recurrent vector, then it is chaotic and therefore frequently hypercyclic.
\end{theorem}

For unilateral shifts, this result is due to Gal\'an et al. \cite[Theorem 2, Corollary 1 and following Remark]{GMOP}; indeed, their statement is more restrictive, but they actually prove the full result, which was also obtained in He et al.~\cite[Corollary 4.2]{HHY18}. A special case is due to Grivaux et al. \cite[Remark 5.21]{GriMaMe}. For bilateral shifts, the result was obtained by the first two authors \cite{BG2}.

The previous theorem can be considerably improved if the underlying space is $\ell^p$. The following is a consequence of \cite{BG2}, using an idea of B\`es et al.~\cite{BMPP16}. Note that it does not hold on $c_0(\mathbb{N})$ by \cite[Theorem 5]{BaRu15}.

\begin{theorem}\label{t-reitshift}
Let $B_w$ be a weighted backward shift on $\ell^p(\N)$ or $\ell^p(\Z)$, $1\leq p<\infty$. If $B_w$ admits a non-zero reiteratively recurrent vector, then it is chaotic and therefore frequently hypercyclic.
\end{theorem}

As we said above we do not know whether, for general Fr\'echet sequence spaces, the existence of a single non-zero frequently recurrent vector, say, implies that the shift is frequently hypercyclic. Luckily, for the main result of this section, we only need the following  weaker implication.

\begin{lemma}
Let $X$ be a Fr\'echet sequence space over $\N$ in which $(e_n)_n$ is an unconditional basis. Suppose that the weighted backward shift $B_w$ is an operator on $X$. If there exists a non-zero vector that is \textit{frequently recurrent} (\textit{upper frequently recurrent}, \textit{reiteratively recurrent}) then there is a set $A\subset\N$ of positive lower density (positive upper density, positive upper Banach density, respectively) such that
\[
\sum_{n\in A} \frac{1}{\prod_{\nu=1}^{n} w_\nu}e_{n} \quad \text{converges in $X$}.
\]
\end{lemma}

This follows from He et al.~\cite[Lemma 5]{HHY18}; see also \cite{BG2}. The lemma allows us to prove the following.

\begin{theorem}\label{t-dist}
\begin{itemize}
\item[\emph{(a)}] There is a hypercyclic operator without non-zero reiteratively recurrent vectors. The operator may even be mixing.
\item[\emph{(b)}] There is a reiteratively hypercyclic operator without non-zero upper frequently recurrent vectors.
\item[\emph{(c)}] There is an upper frequently hypercyclic operator without non-zero frequently recurrent vectors.
\item[\emph{(d)}] There is a frequently hypercyclic operator without non-zero uniformly recurrent vectors.
\end{itemize}
\end{theorem}

In view of Corollary \ref{c-unifnotper} we may complete this list by the following assertion; note that, by Theorem \ref{t-reitshift}, such an operator cannot be a weighted shift on $\ell^p(\N)$ or $\ell^p(\Z)$, $1\leq p<\infty$.
\begin{itemize}
\item[(e)] \textit{There is a hypercyclic uniformly recurrent operator without non-zero periodic points.}
\end{itemize}

\begin{proof}[Proof of Theorem \ref{t-dist}]
(a) By \cite[Theorem 13]{BMPP16}, the weighted backward shift on $\ell^1(\N)$ with weight sequence $w=(\frac{n+1}{n})_n$ is mixing but not reiteratively hypercyclic. In view of Theorem \ref{t-reitshift} it cannot have a non-zero reiteratively recurrent vector.

(b) By \cite[Theorem 7]{BMPP16} (see also \cite[Theorem 7.1]{BG} for a simplified proof) there exists a reiteratively hypercyclic weighted backward shift $B_w$ on $c_0(\N)$ that is not upper frequently hypercyclic. In fact, the proofs show that the weight even satisfies that there is no set $A\subset \N$ of positive upper density such that
\[
\prod_{\nu=1}^{n} w_\nu \to \infty\quad \text{as $n\to\infty$, $n\in A$.}
\]
Thus, by the lemma, $B_w$ cannot have a non-zero upper frequently recurrent vector.

(c) This follows exactly as in (b), using \cite[Theorem 5]{BaRu15} or \cite[Theorem 7.2]{BG} and their proofs.

(d) By \cite[Corollary 5.2]{BaGr07} (see also \cite[Theorem 7.3]{BG}) there exists a frequently hypercyclic weighted backward shift on $c_0(\N)$ that is not chaotic. In view of Theorem \ref{t-rechc}, this operator cannot have non-zero uniformly recurrent vectors.
\end{proof}

\begin{corollary}\label{c-dist}
\begin{itemize}
\item[\emph{(a)}] There is a recurrent operator without non-zero reiteratively recurrent vectors.
\item[\emph{(b)}] There is a reiteratively recurrent operator without non-zero upper frequently recurrent vectors.
\item[\emph{(c)}] There is an upper frequently recurrent operator without non-zero frequently recurrent vectors.
\item[\emph{(d)}] There is a frequently recurrent operator without non-zero uniformly recurrent vectors.
\item[\emph{(e)}] There is a uniformly recurrent operator without non-zero periodic points.

\end{itemize}
\end{corollary}

\section{${\text{IP}}^*$-recurrence}\label{s-iprec}

In the next section we will discuss recurrence properties of further operators. As we will see, a rich supply of eigenvectors to unimodular eigenvalues allows us not only to deduce uniform recurrence for many of these operators, but even a stronger notion that is defined in terms of the so-called $\text{IP}^*$-sets. Before turning to these examples we will therefore study ${\text{IP}}^*$-recurrence in this section.

The starting point is the family $\mathcal{IP}$ of $\text{IP}$-sets. As mentioned in \cite[p. 52]{Fur3}, this family arises naturally when one studies the structure of the sets of integers that can serve as the set of recurrence times for some point in the system. We recall that $A\subset \mathbb{N}_0$ is an $\text{IP}$-set if there exists a strictly increasing sequence $(k_n)_{n\in \mathbb{N}}$ of positive integers such that $k_{j_1}+\dots +k_{j_m}\in A$ whenever $j_1<\dots <j_m$ and $m\in\mathbb{N}$. Then a vector $x\in X$ is called $\text{IP}$-recurrent for an operator $T\in L(X)$ if, for any neighbourhood $U$ of $x$, the return set $N(x,U)$ is an $\text{IP}$-set. But it follows from \cite[Theorem 2.17]{Fur3} that, in our setting, every recurrent vector satisfies this property, so that the notions of recurrence and $\text{IP}$-recurrence coincide.

It is more interesting to study the dual family $\mathcal{IP}^*$, that is, the family of all subsets of $\mathbb{N}_0$ that intersect every set in $\mathcal{IP}$ non-trivially. The elements of this family are called $\text{IP}^*$-sets. A vector $x\in X$ is called $\text{IP}^*$-recurrent for $T\in L(X)$ if, for any neighbourhood $U$ of $x$, the return set $N(x,U)$ is an $\text{IP}^*$-set, see \cite[Chapter 9]{Fur3}. The corresponding set of vectors is denoted by $\text{IP}^*\text{Rec}(T)$. If this set is dense in $X$ then the operator $T$ is called $\text{IP}^*$-recurrent.

It is known that $k\mathbb{N}_0=\{ kn : n\in\mathbb{N}_0\}$ is an $\text{IP}^*$-set for every $k\in\mathbb{N}$ (see the argument in the case $k=2$ in \cite[p. 178]{Fur3}) and that every $\text{IP}^*$-set is syndetic, see \cite[Lemma 9.2]{Fur3}. This implies that
\begin{equation}\label{eq-pipu}
\text{Per}(T)\subset \text{IP}^*\text{Rec}(T)\subset \text{URec}(T),
\end{equation}
and every $\text{IP}^*$-recurrent operator is uniformly recurrent.

The notion of $\text{IP}^*$-recurrence in a linear context has already been studied by Gal\'an et al. \cite {GMOP}. It was observed there that, thanks to a classical result of Furstenberg \cite[Theorem 9.11]{Fur3}, $\text{IP}^*$-recurrent vectors of an operator $T\in L(X)$ coincide with product recurrent vectors, that is, those vectors $x\in X$ such that, for any operator $S\in L(Y)$ on a Fr\'echet space $Y$ and for any recurrent vector $y$ for $S$, the vector $(x,y)$ is recurrent for $T\times S$.

Note also that the vector $x$ with unbounded orbit constructed in Example~\ref{ex_urec_unbounded} or the one in Example~1 of \cite{GMOP} are $\text{IP}^*$-recurrent since, for every neighbourhood $U$ of $x$, we have that $k\N_0\subset N(x,U)$ for some $k\in\N$.

An important property of the family $\mathcal{IP}$ is that it is partition regular, that is, if $A_1\cup A_2\in\mathcal{IP}$, then either $A_1\in\mathcal{IP}$ or $A_2\in\mathcal{IP}$; this immediately implies have that its dual family $\mathcal{IP}^*$ is a filter, see \cite[Lemma 9.5]{Fur3}.

\begin{theorem} Let $T\in L(X)$. Then $\text{\emph{IP}}^*\text{\emph{Rec}}(T)$ is a linear subspace of $X$. In particular, if $T$ is $\text{\emph{IP}}^*$-recurrent, then $T$ admits a dense linear manifold of $\text{\emph{IP}}^*$-recurrent vectors.
\end{theorem}

\begin{proof}
Let $x_1,x_2\in \text{IP}^*\text{Rec}(T)$, and let $\lambda_1, \lambda_2$ be scalars. Given an arbitrary neighbourhood $U$ of $x:=\lambda_1 x_1+\lambda_2 x_2$, we fix neighbourhoods $U_j$ of $x_j$, $j=1,2$, such that
\[
\lambda_1U_1+\lambda_2 U_2\subset U.
\]
Therefore we conclude, by the filter property of  $\mathcal{IP}^*$, that
\[
N(x,U)\supset N(x_1,U_1)\cap N(x_2,U_2)\in \mathcal{IP}^*,
\]
the second part being a consequence of the definition of $\mathcal{IP}^*$-recurrence.
\end{proof}

We next obtain that, for power bounded operators, uniform recurrence and $\text{IP}^*$-recurr\-ence coincide. To do this we need to recall the concept of proximality: Given a dynamical system $(X,T)$, where $(X,d)$ is a metric space, we say that two points $x,y\in X$ are proximal for $T$ if there exists an increasing sequence $(n_k)_{k\in \N}$ of positive integers such that $d(T^{n_k}x,T^{n_k}y)\to 0$ as $k\to\infty$.

\begin{theorem}\label{t-pbdd}
Let $T\in L(X)$. If $T$ is power bounded, then
\[
\text{\emph{IP}}^*\text{\emph{Rec}}(T)=\text{\emph{URec}}(T).
\]
\end{theorem}

\begin{proof}
We just need to show that every uniformly recurrent vector is $\text{IP}^*$-recurrent. Thus let $x\in X$ be uniformly recurrent for $T$. By Theorem \ref{t-power2} the closure $K$ of its orbit is a compact set, and it is $T$-invariant. When we now apply \cite[Theorem 9.11]{Fur3} to the dynamical system $(K, T|_K)$ we see that $x$ is $\text{IP}^*$-recurrent (for $T$) provided that no point $y\neq x$ in $K$ is proximal to $x$.

Thus, suppose that there is some $y\in K$ with $y\neq x$ such that $x$ and $y$ are proximal for $T$. Let $(n_k)_{k\in\N}$ be an increasing sequence of positive integers such that  $d(T^{n_k}x,T^{n_k}y)\to 0$ as $k\to\infty$. Since we may assume that the metric $d$ on $X$ is translation-invariant, we get that
\[
T^{n_k}(x-y)\to 0\quad\text{as $k\to\infty$}.
\]
By equicontinuity of $(T^n)_{n\in\N_0}$ we then have that $T^n(x-y)\to 0$ as $n\to\infty$.

Now, since $x$ is recurrent there is an increasing sequence $(m_k)_{k\in\N}$ of positive integers so that $T^{m_k}x\to x$ as $k\to\infty$. Thus, $T^{m_k}T^nx\to T^nx$ as $k\to\infty$, for each $n\in\N_0$. Again by equicontinuity of $(T^n)_{n\in\N_0}$ and the density of the orbit of $x$ in $K$, we get that $T^{m_k}z\to z$ as $k\to\infty$, for every $z\in K$. In particular,
\[
0=\lim_{k\to\infty} T^{m_k}(x-y)=x-y \neq 0,
\]
which is the desired contradiction.
\end{proof}

The above result suggests the following problem.

\begin{question}\label{q10}
Is there an operator that is uniformly recurrent but not $\text{IP}^*$-recurrent?
\end{question}

Let us comment on this problem.

\begin{remark}
(a) The Le\'on-M\"uller theorem holds for $\textrm{IP}^*$-recurrence, that is, for any operator $T\in L(X)$ and any scalar $\lambda$ with $|\lambda|=1$ we have that $\text{IP}^*\text{Rec}(\lambda T)= \text{IP}^*\text{Rec}(T)$. This is an easy consequence of the fact that a vector is $\text{IP}^*$-recurrent if and only if it is product recurrent, and the fact that the Le\'on-M\"uller theorem holds for recurrence. It thus follows exactly as in the proof of Corollary \ref{c-unifnotper} that there exists a hypercyclic operator on Hilbert space that has a dense set of $\text{IP}^*$-recurrent vectors but no non-zero periodic points. In particular, the first inclusion in \eqref{eq-pipu} can be strict in a very strong sense, but we do not know the status of the second inclusion. Incidentally, $\mathcal{IP}^*$ does not have CuSP, so that one cannot deduce the Le\'on-M\"uller property as in the proof of Theorem \ref{t-LeMu}: the set $A=\N_0$ can be partitioned into the even and the odd numbers, but the set of odd numbers is not an $\text{IP}^*$-set, see \cite[p. 178]{Fur3}.

(b) For all of the operators considered in this paper, whenever we could show uniform recurrence, we even obtained $\text{IP}^*$-recurrence. This will be a common pattern for the operators considered in the next section. And for weighted backward shift operators see Theorem \ref{t-rechc}.
\end{remark}

\section{Recurrence and unimodular eigenvectors}\label{s-receigen}

As promised we now study recurrence properties of various classes of operators. We limit ourselves to operators studied by Costakis et al.~\cite{CMP}; our results strengthen several of their results. In order to keep the paper short we refer to that paper for the definition of the operators and the spaces involved.

It turns out that for practically all of these operators their unimodular eigenvectors play a crucial role. In the sequel we will only consider Fr\'echet spaces over the complex field, and we recall that $\T$ denotes the unit circle in $\C$.

The following result is the key point in this section.

\begin{lemma}\label{l-kron}
Let $\lambda_1,\ldots,\lambda_k\in \T$, $k\geq 1$. Then, for any $\varepsilon>0$,
\[
\{ n\geq 0 : \sup_{j=1,\ldots,k} |\lambda_j^n -1|< \varepsilon\} \in \mathcal{IP}^*;
\]
in particular, it is a syndetic set.
\end{lemma}

\begin{proof}
This is a consequence of \cite[Proposition 9.8 with Lemma 9.2]{Fur3}, applied to the Kronecker system consisting of the compact group $\mathbb{T}^k$ and the (left) multiplication $(z_1,\ldots,z_k)\to (\lambda_1z_1,\ldots,\lambda_kz_k)$.
\end{proof}

We mention that if one is only interested in proving that the sets are syndetic, then one finds a nice proof in \cite[Lemma 3.1]{MT} based on Kronecker's theorem.

For $T\in L(X)$ we denote by
\[
\mathcal{E}(T) = \bigcup_{\lambda\in\T} \text{Eig}(T,\lambda)
\]
the set of unimodular eigenvectors for $T$.

For uniform recurrence, the following was obtained, with a different proof, in \cite[Fact 5.6]{GriMaMe}.

\begin{corollary}\label{c-GMM}
Let $T\in L(X)$. Then
\[
\text{\emph{span}}\,\mathcal{E}(T) \subset \text{\emph{IP}}^*\text{\emph{Rec}}(T).
\]
In particular, if $\text{\emph{span}}\,\mathcal{E}(T)$ is dense in $X$ then $T$ is $\text{\emph{IP}}^*$-recurrent, and hence uniformly recurrent.
\end{corollary}

\begin{proof} Let $x\in \text{span}\,\mathcal{E}(T)$, and let $W$ be a $0$-neighbourhood. We can write $x=\sum_{j=1}^k a_j x_j$ with $a_j\in\C$ and $x_j\in X$ such that $Tx_j=\lambda_jx_j$, $\lambda_j\in \T$, for $j=1,\ldots,k$. Then there is some $\varepsilon>0$ such that $\sum_{j=1}^k \eta_j a_j x_j \in W$ whenever $|\eta_j|<\varepsilon$ for $j=1,\ldots,k$. Now, by the lemma, there is a set $A\in\mathcal{IP}^*$ such that $|\lambda_j^n-1|<\varepsilon$ for all $n\in A$ and $j=1,\ldots,k$. Thus we have for any $n\in A$ that
\[
T^nx- x = \sum_{j=1}^k (\lambda_j^n-1) a_j x_j \in W,
\]
which shows the claim.
\end{proof}

The corollary reminds one of the well-known fact that the set of periodic points of an operator has the representation
\[
\text{Per}(T) = \text{span}\Big(\bigcup_{\substack{\lambda\in\T\\ \lambda \text{ root of unity}}} \text{Eig}(T,\lambda)\Big),
\]
see \cite[Proposition 2.33]{GrPe11}. We will see in Remark \ref{r-strict} below that we do not necessarily have equality in Corollary \ref{c-GMM}.

We start by looking at operators on finite-dimensional spaces.

\begin{theorem}
Let $n\geq 1$. Then, for a matrix $T:\C^n\to \C^n$, the following assertions are equivalent:
\begin{itemize}
\item[\emph{(a)}] $T$ is recurrent;
\item[\emph{(b)}] $T$ is uniformly recurrent;
\item[\emph{(c)}] $T$ is $\text{\emph{IP}}^*$-recurrent;
\item[\emph{(d)}] $T$ is similar to a diagonal matrix with unimodular diagonal entries.
\end{itemize}
In that case, every vector in $\C^n$ is $\text{\emph{IP}}^*$-recurrent for $T$.
\end{theorem}

\begin{proof}
Costakis et al.~\cite[Theorem 4.1]{CMP} have shown that (a) and (d) are equivalent. Thus it suffices to show that (d) implies (c).

Let $S$ be an invertible matrix such that $S^{-1}TS$ is a diagonal matrix with unimodular diagonal entries. Then $Se_k\in \mathcal{E}(T)$, for $k=1,\ldots,n$. Thus $\text{span}\,\mathcal{E}(T)=\C^n$, and we conclude with Corollary \ref{c-GMM}.
\end{proof}

This result suggests to consider general multiplication operators.

\begin{theorem}
Let $X$ be a complex Fr\'echet sequence space over $\N$ which contains $\text{\emph{span}}\{ e_n:n\in\N\}$ as a dense set. Let $(\lambda_n)_n$ be a sequence in $\C$, and let $M_\lambda$ be the multiplication operator
\[
M_\lambda (x_n)_n = (\lambda_n x_n)_n,
\]
which we suppose to be an operator on $X$.

\emph{(A)} The following assertions are equivalent:
\begin{itemize}
\item[\emph{(a)}] $M_\lambda$ is recurrent;
\item[\emph{(b)}] $M_\lambda$ is uniformly recurrent;
\item[\emph{(c)}] $M_\lambda$ is $\text{\emph{IP}}^*$-recurrent;
\item[\emph{(d)}] $\lambda_n\in\T$ for all $n\geq 1$.
\end{itemize}

\emph{(B)} If $M_\lambda$ is power bounded and one of the conditions in \emph{(A)} holds then every vector in $X$ is $\text{\emph{IP}}^*$-recurrent for $M_\lambda$.
\end{theorem}

\begin{proof}
(A) It obviously suffices to show that (d) implies (c). For this, we need only observe that every sequence $e_n$, $n\geq 1$, belongs to $\mathcal{E}(M_\lambda)$, so that $\text{span}\,\mathcal{E}(M_\lambda)$ is dense in $X$ by hypothesis.

(B) This is a direct consequence of Theorems \ref{t-power} and \ref{t-pbdd}.
\end{proof}

We note that if $(e_n)_n$ is an unconditional basis of $X$ and if and one of the conditions in (A) holds then $M_\lambda$ is power bounded, so that the conclusion of (B) holds in this case.

This result has an interesting consequence.

\begin{remark}\label{r-strict}
Let us consider a multiplication operator $M_\lambda$ on $\ell^2(\N)$, say, where the $\lambda_n \in \T$, $n\geq 1$, are pairwise distinct. Then the non-zero multiples of the $e_n$, $n\geq 1$, are the only unimodular eigenvectors, so that $\text{span}\,\mathcal{E}(T)$ contains exactly the finite sequences. On the other hand, by the previous result, every vector in $\ell^2(\N)$ is $\text{IP}^*$-recurrent. This shows that the inclusion in Corollary \ref{c-GMM} may be strict.
\end{remark}

We turn, more generally, to multiplication operators on spaces of measurable functions. If $(\Omega,\mathcal{A},\mu)$ is a measure space, then we call a function $\phi:\Omega\to \C$ essentially countably valued in $D\subset \C$ if there is a countable subset $C\subset D$ such that $\phi(t)\in C$ for $\mu$-almost every $t\in \Omega$.

\begin{theorem}
Let $(\Omega,\mathcal{A},\mu)$ be a measure space, $\phi$ a bounded measurable function on $\Omega$, and let $M_{\phi}$ be the multiplication operator
\[
M_{\phi}f = \phi f
\]
on $L^p(\Omega)$, $1\leq p < \infty$.

\emph{(a)} \emph{(\cite{CMP})} If $M_\phi$ is recurrent then $\phi(t)\in \T$ for $\mu$-almost every $t\in \Omega$.

\emph{(b)} If $\phi$ is essentially countably valued in $\T$ then $M_\phi$ is $\text{\emph{IP}}^*$-recurrent; in fact, every vector in $L^p(\Omega)$ is $\text{\emph{IP}}^*$-recurrent.
\end{theorem}

\begin{proof}
Part (a) was shown in the proof of \cite[Theorem 7.6]{CMP}.

For (b), let $\phi(t)\in \{ \lambda_1, \lambda_2, \ldots \}\subset \T$ for almost every $t\in \Omega$, and let $E_k =\{ t\in \Omega: \phi(t) =\lambda_k\}$, $k\ge 1$. Then $\mu(\Omega\setminus \bigcup_{k\geq 1}E_k)=0$.

Now let $f\in L^p(\Omega)$, $f\neq 0$, and $\varepsilon >0$. There exists some $N\ge 1$ such that
$$
\int _{\bigcup_{k>N}E_k}|f|^pd\mu <\frac{\varepsilon}{2^{p+1}}.
$$
By Lemma \ref{l-kron}, there is a set $A\in\mathcal{IP}^*$ such that, for any $n\in A$ and $k=1,\ldots,N$,
$$
|\lambda_k^n -1|^p <\frac{\varepsilon}{2\|f\|^p}.
$$
Therefore we have for every $n\in A$ that
\begin{align*}
\|M_\phi^nf-f\|^p =& \sum_{k=1}^N\int _{E_k}|\phi^n-1|^p|f|^p\text{d}\mu +\int _{\bigcup_{k>N}E_k}|\phi^n-1|^p|f|^p\text{d}\mu \\
\le &\frac{\varepsilon}{2\|f\|^p}\sum_{k=1}^N\int _{E_k}|f|^p\text{d}\mu+2^p\int _{\bigcup_{k>N}E_k}|f|^pd\mu <\varepsilon.
\end{align*}
Thus, $f$ is $\text{IP}^*$-recurrent.
\end{proof}

We note that in order to only get $\text{IP}^*$-recurrence of the operator in (b) we could have used Corollary \ref{c-GMM}. Indeed, any indicator function $\mathds{1}_E$ lies in $\mathcal{E}(T)$ when $E$ is a measurable subset of some $E_k$, $k\geq 1$. Then clearly $\text{span}\,\mathcal{E}(T)$ is dense in $L^p(\Omega)$.

\begin{example}
Costakis et al.~\cite[Example 7.9]{CMP} consider the case when $\Omega=[0,1]$ with the Lebesgue measure and $\phi:[0,1]\to \T$ given by $\phi(t) = \text{e}^{2\pi i f(t)}$, where $f:[0,1]\to [0,1]$ is the Cantor-Lebesgue function. Then not only is $\phi$ essentially countably valued in $\T$, but almost all of its values are even roots of unity. The previous argument then shows that $M_\phi$ has a dense set of periodic points.
\end{example}

We next turn to composition operators. We first look at operators on $H(\C)$ and $H(\D)$, the Fr\'echet spaces of entire functions and of holomorphic functions on $\D$, respectively, both endowed with the topology of uniform convergence on compact sets.

\begin{theorem} \label{entire}
Let $\phi:\C\to\C$ be a holomorphic function, and let $C_\phi$ be the composition operator on $H(\C)$ given by $C_\phi f = f\circ \phi$. Then the following assertions are equivalent:
\begin{itemize}
\item[\emph{(a)}] $C_{\phi}$ is recurrent;
\item[\emph{(b)}] $C_{\phi}$ is uniformly recurrent;
\item[\emph{(c)}] $C_{\phi}$ is $\text{\emph{IP}}^*$-recurrent;
\item[\emph{(d)}] $\phi(z)=az+b$, $z\in \C$, with $a\in \T$ and $b\in \C$.
\end{itemize}
Moreover, every vector is $\text{\emph{IP}}^*$-recurrent for $C_\phi$ if and only if $\phi(z)=az+b$, $z\in \C$, with $a=1$ and $b=0$, or $a\in\T\setminus \{1\}$ and $b\in \C$.
\end{theorem}

\begin{proof}
By \cite[Theorem 6.4]{CMP} it suffices to show that (d) implies (c), and that the second claim holds. Thus, let $\phi(z)=az+b$, $z\in \C$, with $a\in \T$ and $b\in \C$. If $a=1$ and $b\neq 0$, then $C_\phi$ is well known to be chaotic, see \cite[Example 2.35]{GrPe11}; in that case, $\text{IP}^*\text{Rec}(C_\phi)$ is dense in but not all of $H(\C)$. If $a=1$ and $b=0$ then clearly $\text{IP}^*\text{Rec}(C_\phi)=H(\C)$. Finally let $a\in\T\setminus \{1\}$ and $b\in \C$. Let $f\in H(\C)$, and fix $R>0$ and $\varepsilon>0$. It was shown in the proof of \cite[Theorem 6.4]{CMP} that there is an $\eta>0$ such that if $|a^n-1|<\eta$, $n\geq 0$, then
\[
\sup_{|z|\leq R} |C_\phi^n f(z) - f(z)| < \varepsilon.
\]
Now it follows from Lemma \ref{l-kron} that $\{ n\geq 0 : |a^n-1|<\eta\}\in \mathcal{IP}^*$. This implies that $f$ is $\text{IP}^*$-recurrent. Thus we have again that $\text{IP}^*\text{Rec}(C_\phi)=H(\C)$.
\end{proof}

It is instructive to note that if, once more, one is only interested in obtaining $\text{IP}^*$-recurrence of the operator then this can easily be done with Corollary \ref{c-GMM}. This is trivial if $a=1$ and $b=0$, and well known if $a=1$ and $b\neq 0$, see \cite[Example 2.35]{GrPe11}. Finally, if $a\in\T\setminus \{1\}$, $b\in\C$, then the functions $f_n(z)=(z+\frac{b}{a-1})^n$, $z\in \C$, $n\geq 0$, are eigenvectors for $C_\phi$ with eigenvalue $a^n\in\T$, and they span the set of polynomials, hence a dense subspace of $H(\C)$.

For the unit disk we have the following.

\begin{theorem} \label{disk2}
Let $\phi:\D\to\D$ be a holomorphic function, and let $C_\phi$ be the composition operator on $H(\D)$ given by $C_\phi f = f\circ \phi$. Then the following assertions are equivalent:
\begin{itemize}
\item[\emph{(a)}]  $C_{\phi}$ is recurrent;
\item[\emph{(b)}] $C_{\phi}$ is uniformly recurrent;
\item[\emph{(c)}] $C_{\phi}$ is $\text{\emph{IP}}^*$-recurrent;
\item[\emph{(d)}] either $\phi$ is univalent and has no fixed point, or $\phi$ is an elliptic automorphism.
\end{itemize}
Moreover, every vector is $\text{\emph{IP}}^*$-recurrent for $C_\phi$ if and only if $\phi$ is an elliptic automorphism.
\end{theorem}

\begin{proof}
By \cite[Theorem 6.9]{CMP} it suffices to show that (d) implies (c), and that the second claim holds. If $\phi$ is univalent and has no fixed point then $C_\phi$ is chaotic by \cite[Section 4]{Sha01}; hence $\text{IP}^*\text{Rec}(C_\phi)$ is dense in but not all of $H(\D)$. If $\phi$ is an elliptic automorphism then, by the proof of \cite[Theorem 6.9]{CMP}, $C_\phi$ is conjugate to $\C_{\phi_\lambda}$ for some $\lambda\in\T$, where $\phi_\lambda(z)=\lambda z$, $z\in \D$. It then follows easily from Lemma \ref{l-kron} that $\text{IP}^*\text{Rec}(C_\phi)=H(\D)$.
\end{proof}

We end the section by considering composition operators on the Hardy space $H^2(\D)$.

\begin{theorem} \label{disk}
Let $\phi:\D\to\D$ be a linear fractional map $\phi(z)=\frac{az+b}{cz+d}$, $z\in\D$, with $ad-bc\neq 0$. Let $C_\phi$ be the composition operator on $H^2(\D)$ given by $C_\phi f = f\circ \phi$. Then the following assertions are equivalent:
\begin{itemize}
\item[\emph{(a)}]  $C_{\phi}$ is recurrent;
\item[\emph{(b)}] $C_{\phi}$ is uniformly recurrent;
\item[\emph{(c)}] $C_{\phi}$ is $\text{\emph{IP}}^*$-recurrent;
\item[\emph{(d)}] $\phi$ is either hyperbolic with no fixed point in $\D$, or a parabolic automorphism, or an elliptic automorphism.
\end{itemize}
Moreover, every vector is $\text{\emph{IP}}^*$-recurrent for $C_\phi$ if and only if $\phi$ is an elliptic automorphism.
\end{theorem}

\begin{proof}
By \cite[Theorem 6.12]{CMP} it suffices to show that (d) implies (c), and that the second claim holds. If $\phi$ is hyperbolic  with no fixed point in $\D$, or a parabolic automorphism, then it is chaotic by \cite[Corollary 7]{Hos03}, hence $\text{IP}^*\text{Rec}(C_\phi)$ is dense in but not all of $H^2(\D)$. If $\phi$ is an elliptic automorphism then we conclude as in the previous proof.
\end{proof}

Recurrence properties of further operators can easily be deduced from results in Costakis et al.~\cite[Sections 6, 7]{CMP}.

\section{$\mathcal{F}$-recurrence}\label{s-frec}

In this paper we have concentrated on the most important types of recurrence in order to highlight their differing behaviour. In this section we will briefly study the general notion of $\mathcal{F}$-recurrence, and we consider operators on arbitrary topological vector spaces. The concept was introduced by Furstenberg \cite[Chapter 9]{Fur3} in a non-linear context.

Recall that a non-empty family $\mathcal{F}$ of subsets of $\mathbb{N}_0$ is called a Furstenberg family if $A\in \mathcal{F}$ and $B\supset A$ implies that $B\in \mathcal{F}$; we will assume throughout that $\mathcal{F}$ does not contain the empty set. A Furstenberg family is called left-invariant (right-invariant) if $A\in \mathcal{F}$ and $n\geq 0$ implies that $A-n := \{k-n: k\in A, k\geq n\} \in \mathcal{F}$ (respectively $A+n\in \mathcal{F}$).

\begin{definition}\label{d-frec}
Let $X$ be a topological vector space, $T\in L(X)$, and let $\mathcal{F}$ be a Furstenberg family. Then a vector $x\in X$ is called \textit{$\mathcal{F}$-recurrent} if, for any neighbourhood $U$ of $x$, the return set $N(x,U)$ belongs to $\mathcal{F}$. The set of $\mathcal{F}$-recurrent vectors is denoted by $\mathcal{F}\text{Rec}(T)$. If this set is dense in $X$ then the operator is called \textit{$\mathcal{F}$-recurrent}.
\end{definition}

\begin{remark}\label{r-topfrec}
Costakis et al.~\cite{CMP} have defined $T$ to be recurrent if, for any non-empty open subset $U$ of $X$, the set
\[
N(U,U)=\{ n\geq 0 : T^n(U)\cap U\neq\varnothing\} \text{ is non-empty},
\]
which amounts to demanding that it be in the family of infinite sets. By \cite[Proposition 2.1 with Remark 2.2]{CMP}, this is equivalent to the definition used in this paper provided that $X$ is a Fr\'echet space.

More generally it might be interesting to study the operators $T$ with the following property: for any non-empty open subset $U$ of $X$,
\[
N(U,U)=\{ n\geq 0 : T^n(U)\cap U\neq\varnothing\}\in\mathcal{F}.
\]
Motivated by \cite{CP} one might call these operators \textit{topologically $\mathcal{F}$-recurrent}. This notion is naturally linked to the concept of $\mathcal{F}$-transitive operators as introduced by B\`es et al. \cite{BMPP19}.
\end{remark}

We have preferred the pointwise definition adopted in this paper in order to be close to the corresponding notion of $\mathcal{F}$-hypercyclicity. Recall that an operator $T\in L(X)$ is $\mathcal{F}$-hypercyclic if there is some $x\in X$ such that, for any non-empty open set $V$ in $X$, $N(x,V)\in \mathcal{F}$, see \cite{BMPP16}. The vector $x$ is then called $\mathcal{F}$-hypercyclic.

We have the following generalizations of results in the first part of the paper. The proofs follow as in the special cases, see Theorems \ref{reit-hyp}, \ref{ufreq-hyp}, \ref{t-zero_orbits}, \ref{t-power}, \ref{t-ansari} and \ref{t-LeMu} with Remarks \ref{r-tvs1} and \ref{r-tvs2}.

\begin{theorem}\label{t-hyprec}
Let $X$ be a topological vector space, $T\in L(X)$, and $\mathcal{F}$ a right-invariant Furstenberg family. Then a vector is $\mathcal{F}$-hypercyclic if and only if it is
hypercyclic and $\mathcal{F}$-recurrent.

In particular, $T$ is $\mathcal{F}$-hypercyclic if and only it admits a hypercyclic $\mathcal{F}$-recurrent vector.
\end{theorem}

For the following results we need the concept of an (u.f.i.)~upper Furstenberg family; we refer to \cite{BG} and Theorem 3.1 there.

\begin{theorem}\label{T_upperhyprec}
Let $X$ be a Fr\'echet space, $T\in L(X)$, and $\mathcal{F}$ a right-invariant upper Furstenberg family. Then the following assertions are equivalent:
\begin{itemize}
\item[\rm (a)] $T$ is $\mathcal{F}$-hypercyclic;
\item[\rm (b)] $T$ is hypercyclic, and $\mathcal{F}\text{\emph{Rec}}(T)$ is a residual set;
\item[\rm (c)] $T$ is hypercyclic, and $\mathcal{F}\text{\emph{Rec}}(T)$ is of second category;
\item[\rm (d)] $T$ admits a hypercyclic $\mathcal{F}$-recurrent vector.
\end{itemize}
In that case the set of hypercyclic $\mathcal{F}$-recurrent vectors is residual.
\end{theorem}

\begin{theorem}\label{t-zero_orbits2}
Let $X$ be a Fr\'echet space, $T\in L(X)$, and $\mathcal{F}$ a u.f.i.~upper Furstenberg family. Suppose that there is a dense set of vectors $x\in X$ such that $T^nx\to 0$ as $n\to\infty$. Then $T$ is $\mathcal{F}$-hypercyclic if and only if it is $\mathcal{F}$-recurrent.
\end{theorem}

\begin{theorem}\label{t-power3}
Let $X$ be a topological vector space and $T\in L(X)$. If $T$ is power bounded, then the set $\mathcal{F}\emph{\text{Rec}}(T)$ is closed.
\end{theorem}

The property CuSP for a family of subsets of $\N_0$ was introduced in Section \ref{s-recspec}.

\begin{theorem}\label{t-AnsLeMu}
Let $X$ be a topological vector space, $T\in L(X)$, and $\mathcal{F}$ a Furstenberg family with CuSP.

\emph{(a)} Let $p\geq 1$. Assume that, for any $A\subset\N_0$, $A\in \mathcal{F}$ if and only if $pA\in \mathcal{F}$. Then $T$ and $T^p$ have the same $\mathcal{F}$-recurrent vectors. In particular, if $T$ is $\mathcal{F}$-recurrent then so is $T^p$.

\emph{(b)} Let $\lambda$ be a scalar with $|\lambda|=1$. Then $T$ and $\lambda T$ have the same $\mathcal{F}$-recurrent vectors. In particular, if $T$ is $\mathcal{F}$-recurrent then so is $\lambda T$.
\end{theorem}

Theorems \ref{t-hyprec} and \ref{t-AnsLeMu} have an interesting application to $\mathcal{F}$-hypercyclicity.

\begin{theorem}\label{t-AnsLeMuHyp}
Let $X$ be a topological vector space, $T\in L(X)$, and $\mathcal{F}$ a right-invariant Furstenberg family with CuSP.

\emph{(a)} Let $p\geq 1$. Assume that, for any $A\subset\N_0$, $A\in \mathcal{F}$ if and only if $pA\in \mathcal{F}$. Then $T$ and $T^p$ have the same $\mathcal{F}$-hypercyclic vectors. In particular, if $T$ is $\mathcal{F}$-hypercyclic then so is $T^p$.

\emph{(b)} Let $\lambda$ be a scalar with $|\lambda|=1$. Then $T$ and $\lambda T$ have the same $\mathcal{F}$-hypercyclic vectors. In particular, if $T$ is $\mathcal{F}$-hypercyclic then so is $\lambda T$.
\end{theorem}

\section*{Acknowledgements}

We would like to thank the referee, whose careful reading and valuable observations have led to an improvement of the presentation of the article.

~\\[2mm]

\noindent
\parbox[t][3cm][l]{7cm}{\small
\noindent
Antonio Bonilla\\
Departamento de An\'alisis Matem\'atico\\
Universidad de La Laguna\\
C/Astrof\'{\i}sico Francisco S\'anchez, s/n\\
38721 La Laguna, Tenerife, Spain\\
E-mail: a.bonilla@ull.es}
\parbox[t][3cm][l]{8cm}{\small
Karl-G. Grosse-Erdmann\\
Département de Mathématique\\
Universit\'e de Mons\\
20 Place du Parc\\
7000 Mons, Belgium\\
E-mail: kg.grosse-erdmann@umons.ac.be}
\parbox[t][3cm][l]{7cm}{\small
\noindent
Antoni L\'opez-Mart\'{\i}nez\\
Institut Universitari de Matem\`atica\\ Pura i Aplicada\\
Universitat Polit\`ecnica de Val\`encia\\
Edifici 8E, Acces F, 4a planta\\
46022 Val\`encia, Spain\\
E-mail: anlom15a@upv.es}
\parbox[t][3cm][l]{8cm}{\small
Alfred Peris\\
Institut Universitari de Matem\`atica\\ Pura i Aplicada\\
Universitat Polit\`ecnica de Val\`encia\\
Edifici 8E, Acces F, 4a planta\\
46022 Val\`encia, Spain\\
E-mail: aperis@mat.upv.es}

\end{document}